\documentclass[twoside]{article}

\usepackage[accepted]{arxiv}

\usepackage[round]{natbib}

\newcommand{\stepsize}{\eta}		
\newcommand{\varsym}{\sigma}		
\newcommand{\iterationIndex}{t}		
\newcommand{\fin}{1}				
\newcommand{\li}{T}					
\newcommand{\momentum}{\beta}		
\newcommand{\clipThresh}{\gamma}	
\newcommand{\clippedEstt}{\widehat{\nabla} \ora \pare{\xt, \xit}}

\newcommand{\di}{d}		
\newcommand{\obj}{F}	
\newcommand{\ora}{f}		
\newcommand{\sm}{L}

\newcommand{\nameeqref}[1]{\hyperref[#1]{(\nameref{#1})}}
\newcommand{\lb}{\nameeqref{assum:lower_bounded}}
\newcommand{\lsmooth}{\nameeqref{assum:l_smooth}}
\newcommand{\pBCM}{\nameeqref{assum:pBCM}}

\newcommand{\sgdName}{\algname{SGD}}
\newcommand{\nsgdName}{\algname{NSGD}}
\newcommand{\clippedName}{\algname{Clip-SGD}}
\newcommand{\nsgdmName}{{\algname{NSGD-M}}}
\newcommand{\batchNsgdName}{\algname{minibatch-NSGD}}

\newcommand{\sgd}{\sgdName}
\DeclareRobustCommand{\nsgd}{%
	\begingroup\hypersetup{hidelinks}\hyperref[eq:nsgd]{\nsgdName}\endgroup%
	}
\let\clippedSGD\clippedName
\DeclareRobustCommand{\batchNsgd}{%
	\begingroup\hypersetup{hidelinks}\hyperref[eq:batch_oracle]{\batchNsgdName}\endgroup%
	}
\DeclareRobustCommand{\nsgdm}{%
	\begingroup\hypersetup{hidelinks}\hyperref[eq:app.nsgdm]{\nsgdmName}\endgroup%
}

\newcommand{\nob}{\nabla \obj}
\newcommand{\nor}{\nabla \ora}			

\usepackage[utf8]{inputenc} 	
\usepackage[T1]{fontenc}    	

\usepackage{amsmath, amsfonts, amsthm, amssymb}
\usepackage{thmtools}
\usepackage{mathtools}							
\usepackage{nicefrac}       					
\usepackage[only,llbracket,rrbracket]{stmaryrd}	

\usepackage[algo2e, ruled]{algorithm2e}

\usepackage{booktabs}  		     	
\usepackage{graphicx}				
\usepackage{enumitem}				
\usepackage{xspace}					
\usepackage{float}					
\usepackage{caption, subcaption}	
\usepackage[
	colorinlistoftodos,
	bordercolor=orange,
	backgroundcolor=orange!20,
	linecolor=orange,
	textsize=scriptsize]
{todonotes}
\usepackage{tikz,pgfplots}
\pgfplotsset{compat=newest}
\usetikzlibrary{calc}      
\usepackage{makecell}

\usepackage{url}            
\usepackage{xcolor}         
\usepackage{hyperref}       
\colorlet{linkcolour}{blue!50!black}
\hypersetup{
	colorlinks 			= true, 		
	urlcolor   			= linkcolour, 	
	linkcolor  			= linkcolour, 	
	citecolor  			= linkcolour, 	
	bookmarksnumbered 	= true			
}
\usepackage{cleveref}		

\usepackage{preamble}				
\usepackage[opt]{commands}			

\newcommand{\algname}[1]{\texttt{#1}\xspace}

\DeclareMathOperator{\polylog}{polylog}

\genCommandsName{phi}{Scalar}{\phi}
\genCommandsName{psi}{Scalar}{\psi}
\genCommandsName{sig}{Scalar}{\sigma}
\genCommandsName{beta}{Scalar}{\beta}
\genCommandsName{clip}{Scalar}{\clipThresh}
\genCommandsName{eps}{Vector}{\eps}
\genCommands{D}{Scalar}
\genCommands{B}{Scalar}
\genCommands{S}{Vector}

\newcommand{\tauSum}{\sum_{\tau = \fin}^\li}
\newcommand{\sstau}{\iterationss \tau}

\newcommand{\nnFxt}{\norm{\nFxt}}
\newcommand{\xfin}{\iterationx \fin}
\newcommand{\momprod}[2]{\momentum_{#1 : #2}}
\newcommand{\logood}{\log\pare{\nicefrac 1 \delta}}
\newcommand{\Aa}{{ \ssi \sm }}
\newcommand{\tinN}{{ \ii \in \N }}

\renewcommand{\P}{\mathbb{P}}

\theoremstyle{plain}
\newtheorem{theorem}{Theorem}
\newtheorem{lemma}[theorem]{Lemma}
\newtheorem{corollary}[theorem]{Corollary}

\newtheorem{proposition}[theorem]{Proposition}

\newtheorem*{example*}{Example}

\newtheorem{definition}{Definition}
\newtheorem{assum}[definition]{Assumption}
\newtheorem*{assum*}{Assumption}

\newtheoremstyle{boldremark}
{\topsep} 		
{\topsep} 		
{\normalfont} 	
{}          	
{\bfseries} 	
{.}         	
{.5em}      	
{}          	
\theoremstyle{boldremark}

\newtheorem*{remark*}{Remark}

\newtheorem*{notation*}{Notation}

\newtheorem*{setup*}{Problem Setup}			

\begin{document}

\renewcommand*{\thefootnote}{\fnsymbol{footnote}}
\twocolumn[
\aistatstitle{From Gradient Clipping to Normalization for Heavy Tailed SGD}
\aistatsauthor{ Florian Hübler$^*$ \And Ilyas Fatkhullin$^*$ \And  Niao He }
\aistatsaddress{ ETH Zurich \And ETH Zurich \And ETH Zurich } ]
\renewcommand*{\thefootnote}{\arabic{footnote}}

\begin{abstract}
	Recent empirical evidence indicates that many machine learning applications involve heavy-tailed gradient noise, which challenges the standard assumptions of bounded variance in stochastic optimization. Gradient clipping has emerged as a popular tool to handle this heavy-tailed noise, as it achieves good performance both theoretically and practically. 
	However, our current theoretical understanding of non-convex gradient clipping has three main shortcomings. First, the theory hinges on large, increasing clipping thresholds, which are in stark contrast to the small constant clipping thresholds employed in practice. Second, clipping thresholds require knowledge of problem-dependent parameters to guarantee convergence. Lastly, even with this knowledge, current sample complexity upper bounds for the method are sub-optimal in nearly all parameters. 
	To address these issues and motivated by practical observations, we make the connection of gradient clipping to its close relative --- Normalized \texttt{SGD} (\texttt{NSGD}) --- and study its convergence properties.
	First, we establish a parameter-free sample complexity for \texttt{NSGD} of $\mathcal{O}\left(\varepsilon^{-\frac{2p}{p-1}}\right)$ to find an $\varepsilon$-stationary point, only assuming a finite $p$-th central moment of the noise, $p\in(1,2]$. 
	Furthermore, we prove the tightness of this result, by providing a matching algorithm-specific lower bound. In the setting where all problem parameters are known, we show this complexity is improved to $\mathcal{O}\left(\varepsilon^{-\frac{3p-2}{p-1}}\right)$, matching the previously known lower bound for all first-order methods in all problem dependent parameters. 
	Finally, we establish high-probability convergence of \texttt{NSGD} with a mild logarithmic dependence on the failure probability. Our work complements the studies of gradient clipping under heavy-tailed noise, improving the sample complexities of existing algorithms and offering an alternative mechanism to achieve high-probability convergence.
\end{abstract}

\def\thefootnote{*}\footnotetext{These authors contributed equally to this work.}\def\thefootnote{\arabic{footnote}}

\section{INTRODUCTION}
\label{sec:intro}
We study the stochastic optimization problem
\begin{align}\label{eq:problem}
	\min_{x\in \R^d} \obj(x) ,  \qquad \obj(x) := \Expu{\xi \sim \cD}{f(x, \xi)}, 
\end{align}
where $F\colon \R^d \rightarrow \R$ is a potentially non-convex, $\sm$-smooth objective function, and $\xi$ is a random variable with an unknown distribution $\cD$. Such problems are pervasive in machine learning applications, where obtaining exact gradients is often infeasible, necessitating reliance on stochastic gradients \citep{bottou2018optimization}.

Traditionally, stochastic gradient methods rely on the assumption that the variance of the gradient noise is bounded. Under this assumption, it is well established that first-order algorithms require at least $\Omega\pare{\Dz\sm \eps^{-2} + \Dz\sm\varsym^2 \eps^{-4}}$ stochastic gradient oracle queries in the worst case to find an $\eps$-stationary point, i.e., $x\in\R^\di$ with $\Expnorm{\nF\pare x}\leq \eps$ \citep{arjevani2023lower}. Here $\Dz$ denotes the initialization gap $\Dz \geq \obj(\xfin) - \inf_{x \in \R^d} \obj(x)$, $L$ the smoothness-parameter and $\varsym^2$ the variance. Stochastic Gradient Descent (\sgdName) with an appropriately chosen step-size achieves this optimal sample complexity \citep{ghadimi2013stochastic}.

However, new insights in machine learning suggest that the bounded variance (BV) assumption may be overly restrictive. Empirical evidence from fields such as image classification \citep{simsekli2019tail,battash2024revisiting}, training large language models (LLMs) \citep{zhang2020adaptive,ahn2023linear}, and policy optimization in reinforcement learning (RL) \citep{garg2021proximal} indicates that stochastic gradients often follow heavy-tailed distributions. These findings challenge the standard assumption, suggesting a shift towards weaker noise models which only assume boundedness of the $p$-th central moment of the gradient noise for some $p\in (1,2]$, i.e.
\begin{align}
    \Exp{\norm{\nf{x, \xi} - \nabla F(x)}^{p}} \leq \sigma^p \tag{p-BCM}
\end{align}
with $\sigma = \sigma_p \geq 0$, where $p$ denotes the tail index. Specifically, the aforementioned works verify the tail index of stochastic gradients using statistical tests (e.g., \citet{mohammadi2015estimating}) and find $p<2$. Even when the bounded variance assumption holds, the resulting constant $\sigma_2$ can be prohibitively large compared to $\sigma_p$ for some $p<2$.

While \sgd achieves the optimal sample complexity under finite variance, empirical evidence suggests that adaptive algorithms become crucial in the presence of heavy tailed noise \citep{zhang2020adaptive}. The vast majority of works\footnote{Except, for example, \citep{Wang_SGD_inf_var_21}, which studies convergence of \sgd\ in the strongly convex case under additional $p$-positive definiteness assumption on the Hessian.} which are able to prove convergence under \pBCM\ employ the gradient clipping mechanism \citep{zhang2020adaptive,gorbunov2020stochastic,cutkosky2021high,sadiev2023high,gorbunov2023high,Clipped_AdaGrad2023Li,nguyen2023improved,kornilov2024accelerated,HP_NonSmooth_NonConvex_HT2024Liu}. This mechanism replaces the stochastic gradient in optimization algorithms by its clipped counterpart
\begin{align}\label{eq:intro.clipping}
	\clippedEstt = \min\set{1, \frac{\clipt}{\norm{\nf{\xt,\xit}}}} \nf{\xt,\xit},
\end{align}
where $\cb{\clipt}_{\ii \geq 1}$ is the sequence of clipping thresholds and $\xit \iid \cD$. 

Perhaps, the most popular scheme is \clippedName,\footnote{Many variants and modifications of \clippedName exist including its combinations with Nesterov's acceleration \citep{gorbunov2020stochastic}, normalization \citep{cutkosky2021high}, zero-order \citep{kornilov2024accelerated} and coordinate-wise variants \citep{zhang2020adaptive}, but gradient clipping is the key building block of these methods.} which updates the iterates as $\xtp = \xt - \sst \clippedEstt$, where $\cb{\sst}_{\ii \geq 1}$ is a predefined sequence of step-sizes. 


\subsection{Drawbacks of Gradient Clipping Theory}\label{subsec:drawbacks_clipoing}

Despite its popularity in the literature, we want to outline several drawbacks of current clipping theory.
 
\paragraph{Misalignment between theoretical and practical insights.} 

Existing theoretical analyses of \clippedName\ (and its variants) under the \pBCM\ assumption hinge on using a large, $p$-dependent sequence of increasing clipping thresholds (e.g., $\clipt = \gamma \cdot \ii^{\frac 1 {3p-2}}$) \citep{zhang2020adaptive,cutkosky2021high,Clipped_AdaGrad2023Li,nguyen2023high,nguyen2023improved}.
This choice of clipping thresholds is based on the following two ideas. First, clipping allows one to control the variance of the clipped gradient estimator $\clippedEstt$, even in cases where the original gradient oracle has infinite variance. Second, it ensures that the probability of gradients being clipped decreases over time as $\gamma_t$ increases, thereby reducing the bias introduced by clipping and facilitating convergence. However, this theoretical recommendation contradicts common practice for clipping in machine learning, where small, constant thresholds (e.g., $\clipt \equiv 0.25$) are typically used instead \citep{LSTM2018Merity,OPT_model2022Zhang,Llama2023Touvron,DeepseekV32024Liu}.

In contrast, one can observe that the clipping thresholds commonly used in practice lead to an \emph{increasing} probability of clipping gradients, eventually resulting in gradients being clipped at every iteration. This observation runs counter to theoretical insights, which suggest clipping is becoming less frequent as training progresses. Specifically, we observe this phenomenon in language modelling tasks in \Cref{sec:experiments}, and notice the same effect on simpler, synthetic examples in \Cref{sec:app.add_experiments}. This aggressive clipping behaviour essentially transforms \clippedName\ into a variant of Normalized \sgdName:
\begin{align}\label{eq:nsgd}
	\xtp = \xt - \sst \frac{\gt}{\norm{\gt}}, \tag{\nsgdName} 
\end{align}
where $\gt = \nf{\xt,\xit}$ in this case. However, it should be noted that unlike \clippedName, \nsgd\ only requires tuning a single parameter $\eta_t$, highlighting its simplicity in comparison.

\paragraph{Need for tuning.}
To our knowledge, all existing convergence results for gradient clipping under \pBCM\ require knowledge of all problem parameters to set the clipping thresholds $\cb{\clipt}_{t\geq 1}$ and other hyper-parameters of the underlying algorithms. As these problem-dependent parameters are not known in practice, this necessitates an extensive hyper-parameter tuning. In particular, for \clippedSGD, there are $2$ hyper-parameters which potentially require tuning. In \Cref{sec:app.add_experiments} we observe that even in simple scenarios, tuning both parameters may be needed to match the performance of \nsgd, which only has $1$ parameter. Furthermore, in \Cref{sec:experiments} we empirically observe that even while requiring extensive hyper-parameter tuning, \clippedSGD\ is not able to outperform vanilla \nsgd\ in language modeling tasks.

\paragraph{Suboptimal sample complexities.} 
None of the existing convergence analysis of non-convex \clippedName (and its variants) achieve the sample complexity lower bound by \citet{zhang2020adaptive}, $\Omega\pare{\frac{\Dz\sm}{\eps^2} + \frac{\Dz\sm}{\eps^2}\pare{\frac{\varsym}{\eps}}^{\frac{p}{p-1}}}$, in all problem parameters, even when problem parameters are known. In particular, prior to this work, the optimal heavy-tailed sample complexity remained an open question.


\subsection{Our Contributions}
\label{sec:intro.contributions}

Our work seeks to remove the drawbacks listed above by diving into the convergence analysis of \nsgd\ with different gradient estimators\footnote{Our results in the main body are stated for \batchNsgd, the corresponding results for \nsgd\ with momentum can be found in \Cref{sec:app.nsgdm}.} under heavy-tailed noise. We summarize our contributions as follows:

$\,1.\,$ We prove in-expectation convergence of \nsgd\ using either mini-batches or momentum under the \pBCM\ assumption for $p \in (1,2]$ in two settings.
\begin{enumerate}[label=\alph*),topsep=1mm]
    \item Without any knowledge of problem specific parameters, we show that the algorithms require at most $\Oc \pare{\frac{\Dz^4 + \sm^4}{\eps^{4}} + \pare{\frac{\varsym}{\eps}}^{\frac{2p}{p-1}}}$ stochastic gradient oracle queries to reach an $\eps$-stationary point in expectation, providing the first parameter-free heavy-tailed convergence guarantee. Furthermore, we construct an algorithm-specific lower bound showing that this sample complexity is tight for \nsgd\ with polynomial step-size and batch-size.
    \item When problem parameters are known, we achieve a better sample complexity of at most $\Oc \pare{\frac{\Dz \sm}{\eps^2} + \frac{\Dz \sm}{\eps^2} \pare{\frac{\varsym}{\eps}}^{\frac{p}{p-1}}}$, improving the previously best-known sample complexity under heavy-tailed noise. This sample complexity exactly matches the mini-max lower bound in all parameters for the class of first-order algorithms under our assumptions.
\end{enumerate}
To our knowledge, \nsgd\ is the first algorithm which achieves either a) or b) in the heavy tail regime $p<2$.

$\,2.\,$ We provide a high probability convergence guarantee for \batchNsgd, removing the need for clipping, thereby extending our understanding of high-probability guarantees under heavy-tailed noise. The sample complexity in this case corresponds to the same complexity as its in-expectation counterpart with a mild multiplicative $\polylog \pare{\nicefrac 1 \delta}$ factor. Finally, we provide new insights on the convergence measure of non-convex \nsgd. 


\subsection{Related Work}

\textbf{Gradient clipping} is widely used to stabilize the training in various fields of machine learning \citep{pascanu2013difficulty,schulman-et-al17ppo,zhang2020adaptive}. Recently a number of works provide convergence guarantees for \clippedName and its variants in different settings, e.g., \citep{nazin2019algorithms,gorbunov2020stochastic,davis2021low,gorbunov2023high,liu2023stochastic,puchkin2024breaking} to name a few. However, the results in the non-convex stochastic setting are relatively scarce. In particular, \citet{zhang2020adaptive} study in-expectation and \citet{sadiev2023high,nguyen2023improved} investigate high probability convergence of \clippedName under \pBCM. All above mentioned works use increasing (iteration dependent) clipping parameters, e.g., $\gamma_t = \gamma \cdot t^{\frac{1}{3p-2}}$, and derive suboptimal convergence rates, see \Cref{sec:main.expectation} for a more detailed discussion. A momentum version of \clippedName was analyzed in \citep{mai2021stability} assuming the bounded second moment of stochastic gradients. However, their proof crucially relies on setting the clipping threshold larger than the expected gradient norm. Recently, \citet{koloskova2023revisiting} offer a new analysis of \clippedName with a constant clipping threshold under BV setting. However, their proof crucially relies on bounded variance and seems challenging to extend to \pBCM \, setting. It is worth mentioning that gradient clipping is also used to tackle heavy-tailed noise in bandits and RL literature, e.g., \citep{bubeck2013bandits,cayci2024provably} and in online learning \citep{zhang2022parameter}. Moreover, \clippedName is the key mechanism to ensure differential privacy \citep{abadi2016deep,sha2024clip}. 

\textbf{Normalized SGD}
was first proposed by \citet{nesterov1984minimization,nesterov2018lectures} and analyzed in the deterministic convex case. Later the analysis was extended to  smooth \citep{levy2017online} and stochastic \citep{hazan2015beyond} settings. In the non-convex case, \citet{MomentumImprovesNSGD_Cutkosky_2020} show how to remove large mini-batch requirement for \nsgd\ by incorporating Polyak's momentum. Later, \citet{yang2023two_sides} derive a tight lower bound for \nsgd\ without momentum and \citet{NSGDM_LzLo2023Huebler} study the parameter agnosticity of momentum \nsgd\ under a relaxed smoothness assumption. In a different line of work, \citet{levy2016power} study the ability of \nsgd\ to escape from saddle points. However, all above-mentioned works make strong noise assumptions such as BV. The most closely related to our work are \citep{cutkosky2021high,UnboundedClippedNSGDM2023Liu}, which study variants of \nsgd\ under heavy-tailed noise. Unfortunately, these works use both normalization and gradient clipping with increasing clipping parameter, which necessitates tuning their clipping thresholds. Moreover, \citet{cutkosky2021high} assume bounded non-central moment assumption, i.e., $\Exp{\norm{\nabla f(x, \xi)}^p} \leq G^p$, which is stronger than our \pBCM. This assumption is relaxed in \citep{UnboundedClippedNSGDM2023Liu} to \pBCM\ at the cost of imposing an additional (almost sure) individual smoothness assumption for each $f(x, \xi)$. Another line of work assumes that the noise distribution has a probability density function (PDF) that is symmetric and strictly positive in a neighborhood of zero \citep{polyak1979adaptive,jakovetic2023nonlinear,NonlinearHP2023Armacki}. Under this assumption, they study SGD type methods with general non-linearities, which include gradient clipping and normalization as a special case. Compared to these works, we work with a different \pBCM\ assumption. 

More recently, the role of normalization was investigated for sharpness aware minimization \citep{dai2024crucial}, and the variants of \nsgd\ showed impressive empirical and theoretical success in more structured non-convex problems in RL \citep{SPGM_fatkhullin23a,barakat2023reinforcement,ganesh2024variance}. However, these works are also restricted to benign BV noise assumption. Some recent works also make connections with \algname{SignSGD} algorithm \citep{bernstein2018signsgd,Karimireddy_SignSGD,crawshaw2022robustness}, which applies a coordinate-wise normalization. Indeed, the convergence analysis of \algname{SignSGD} and \nsgd\ are closely related and our techniques can be extended to its sign variants \citep{liu2019signsgd,sun2023momentum}.

In a concurrent work, \citet{liu2024nonconvex} derive similar in-expectation upper bounds for \nsgdm\ under heavy-tailed noise, albeit with more general $(\sigma, \sigma_1)$-affine variant of \pBCM\ and $(L, L_1)$-smoothness. In comparison to their work, we additionally study the tightness of our rates by designing a non-trivial algorithm-specific lower bound, establish high probability convergence of \nsgd\, and provide insights on the convergence measure of  \nsgd. While \cite{liu2024nonconvex} also note the parameter-free convergence of \nsgdm, their result is parameter-free only in the special case $\sigma_1 = L_1 = 0$, the setting which recovers our assumptions. Another concurrent work \citep{Concurrent2_2024Sun} examines the convergence of \nsgdm\ under a related but stronger noise assumption, which implicitly ensures bounded noise and thus limits their setting to light-tailed distributions.

\section{PRELIMINARIES}
\label{sec:prelims}
Let us introduce basic notations, definitions and assumptions needed in the upcoming analysis.

\begin{notation*}
	We adopt the common conventions $\N = \set{0, 1, \dots}$, $[n] = \set{1, 2, \dots, n}$ and that empty sums and products are given by their corresponding neutral element. Throughout this paper, $\di\in\Ngeq$ denotes the dimension of the variable to be optimized, $\obj\colon \R^\di \to \R$ the objective and $\nor\pare{\cdot, \cdot}$ the stochastic gradient oracle. Unless stated otherwise, $\sm\geq0$ denotes the $\sm$-smoothness parameter, and $\sst > 0$ is the stepsize. We use the standard $\Oc\pare{\cdot}, \Omega\pare{\cdot}, \omega\pare{\cdot}$ complexity notations \citep{MultivariateO}, $\Oct\pare{\cdot}$ additionally hides poly-logarithmic factors.
\end{notation*}

\begin{setup*}
	Since solving $\min_{x\in\R^d}\obj(x)$ to global optimality is computationally intractable\ \citep{nemirovskij1983problem}, our goal instead is to find an $\eps$-stationary point, i.e., $x\in \R^d$ such that $\norm{\nabla F(x)}\leq \varepsilon$ in expectation or with high probability. 
    Furthermore, we assume the access to first order information is limited to a (potentially noisy) gradient oracle,  $\nor \pare{\cdot, \xi}$ of $\nF$, where $\xi$ is a random variable. The sample complexity is defined as the number of calls the algorithm makes to this oracle to find an $\eps$-stationary point.
\end{setup*}

Throughout the paper we work under the following standard assumptions.

\begin{assum}[Lower Boundedness]\label{assum:lower_bounded}
	The objective function $F$ is lower bounded by $F^* > -\infty$ and we denote $\Delta_1 \geq \obj \pare{\xfin} - F^*$, an upper bound on the initialization gap.
\end{assum}

\begin{assum}[$\sm$-smoothness]\label{assum:l_smooth}
	The objective function $F$ is $\sm$-smooth, i.e.\ $\obj$ is differentiable and for all $x,y\in\R^d$ we have $\norm{\nob(x) -\nob(y)} \leq \sm \norm{x-y}$.
\end{assum}

Instead of the classical bounded variance assumption, we adopt the weaker concept of the bounded $p$-{th} central moment, as discussed in the introduction.

\begin{assum}[$p$-BCM]\label{assum:pBCM}
	The gradient oracle is unbiased and has a finite $p$-{th} central moment, i.e.\ there exists $\varsym_p \geq 0$ such that, for all $x \in \R^d$,
	\begin{enumerate}[label=\roman*),topsep=0pt]
		\item $\Exp{\nor \pare{x, \xi}} = \nob(x)$, and
		\item $\Exp{\norm{\nor \pare{x, \xi} - \nF \pare x}^p} \leq \varsym_p^p$.
	\end{enumerate}
\end{assum}

In this work, we focus on the case $p \in (1,2]$. It is worth noting that, by Jensen's inequality, any oracle satisfying \pBCM\ also satisfies the assumption for all $p' \leq p$, with $\sigma_{p'} \leq \sigma_p$. Notably, \pBCM\ is weaker than the bounded variance assumption, and it is possible for $\sigma_{p'}$ to be much smaller than $\sigma_p$. 
We will omit the subscript throughout the work to improve readability, though the dependence of $\sigma$ on $p$ remains important to keep in mind.

\section{MAIN RESULTS}
\label{sec:main}
In this section, we present our convergence results for normalized stochastic gradient methods under the \pBCM\ assumption. In order to guarantee a consistent presentation, we will present the results for \batchNsgdName,\footnote{We furthermore present the results for known horizon ($\li$-dependent) parameters. Note that all convergence guarantees also hold for decaying ($\ii$-dependent) parameters at the mild cost of a multiplicative $\log\pare\li$ term.} i.e., \nsgd\ with the mini-batch gradient estimator
\begin{align}\label{eq:batch_oracle}
	\gt = \frac 1 {\Bt} \sum_{j=1}^{\Bt} \nabla \ora \pare{\xt, \xit^{(j)}},
\end{align}
where $\xit^{(1)}, \dots, \xit^{(\Bt)}$ are independent copies of $\xit$. All results presented in this section (except for \Cref{cor:main.optimal_hp}) are also derived for \nsgd\ with momentum, i.e., \nsgd\ with the momentum gradient estimator
\begin{align*}
	\iterationg \fin &= \nf{\xfin, \xi_\fin},\\
	\gt &= \momt \gtm + \pare{1-\momt}\nf{\xt, \xit}
\end{align*}
for a momentum parameter $\pare{\momt}_{t \geq 1}$, and are presented in \Cref{sec:app.nsgdm}. Furthermore, for vanilla \nsgd\ (i.e. using $\gt = \nf{\xt, \xit}$), the results imply convergence to a $\varsym$-neighbourhood, in line with the corresponding algorithm specific lower bound \citep[Theorem 3]{yang2023two_sides}. 

In \Cref{sec:main.expectation} we first examine the convergence of \batchNsgd\ for unknown problem-parameters, providing a parameter-free convergence guarantee under the \pBCM\ assumption. Afterwards, we examine the performance for optimally tuned parameters. In \Cref{sec:main.hp}, we derive a high-probability convergence result for \batchNsgd. Finally, in \Cref{sec:main.improve_measure}, we examine the importance of different convergence measures for our analysis.


\subsection{Normalized \texorpdfstring{\sgdName}{SGD} can Handle Heavy Tailed Noise}
\label{sec:main.expectation}

We first theoretically confirm the robustness of \batchNsgd, by providing a parameter-free convergence guarantee. This is in stark contrast to current \clippedSGD\ analyses, which hinge on the knowledge of all parameters.

\begin{proposition}[Simplified]\label{prop:main.param_agnostic_general}
	Assume \lb, \lsmooth\ and \pBCM\ with $p \in (1, 2]$. Let $q \geq 0, \ssi, B > 0$ and $r \in (0,1)$. Then the iterates generated by \batchNsgd\ with parameters $\sst \equiv \ssi \li^{- r}$ and $\Bt \equiv \ceil{\max\set{1, B\li^q}}$ satisfy
	\begin{align*}
		\frac 1 \li \iSum \Expnorm{\nabla_\ii} \leq 
		\frac{\Dz}{\ssi \li^{1-r}} + \frac{\ssi \sm}{2\li^r} + \frac{4 \varsym}{\max\set{1,B\li^q}^{\frac{p-1}{p}}},
	\end{align*}
	where $\nabla_\ii \coloneqq \nFxt$. The sample complexity is bounded by $\Oc \pare{\pare{\frac \Dz \eps}^{\frac{1+q}{1-r}} + \pare{\frac \sm \eps}^{\frac{1+q}{r}} + \pare{\frac{\varsym}{\eps}}^{\frac{p\pare{1+q}}{q\pare{p-1}}}}$.
\end{proposition}

This result characterises the sample complexity of \batchNsgd\ under the \pBCM\ assumption for different orders of step-sizes and batch-sizes. 
Recall that any stochastic gradient oracle satisfying \pBCM, also satisfies the assumption for all $p' \leq p$ with $\varsym_{p'} \leq \varsym_p$. In particular, it is possible that $\varsym_{p'} \ll \varsym_p$ and applying \Cref{prop:main.param_agnostic_general} with $p'$ may yield a smaller sample complexity. Hence the result also implies a potentially better sample complexity bound for a specific oracle by taking the infimum over all $p' \in (1,p]$ of our result.

The proof of \Cref{prop:main.param_agnostic_general} can be found in \Cref{sec:app.missing.exp.prop_1} and follows a similar structure to the case when $p=2$, though it demands additional attention to the noise term. Notably, we employ a vectorized version of the von Bahr and Esseen inequality \citep{von1965inequalities} (see \Cref{lem:app.von_Bahr_and_Essen} in \Cref{sec:app.technical_results}), which provides a more general foundation compared to the ad-hoc approach using additional gradient clipping \citet{cutkosky2021high}, who analyzed \nsgd\ with momentum and gradient clipping. In the special case of $p=2$, our \Cref{prop:main.param_agnostic_general} can recover the previous rates for \nsgd\ in \citep{MomentumImprovesNSGD_Cutkosky_2020}. The extended result, including the dependence on $\ssi$ and $B$, can be found in \Cref{eq:app.missing.exp.param_free.rate,eq:app.missing.exp.param_free.sample_complexity}.

\paragraph{Tightness of \Cref{prop:main.param_agnostic_general}.} While lower bounds on the sample complexity for general first-order algorithms are well-established \citep{arjevani2023lower,zhang2020adaptive}, there are no algorithm-specific lower bounds specifying the optimal oracle complexity of \batchNsgd\ with general parameters. As a consequence, it is unclear whether the parameter-dependence --- in particular the dependence on $r$ and $q$ --- in \Cref{prop:main.param_agnostic_general} is tight. To address this, we establish an algorithm-specific lower bound, demonstrating that \Cref{prop:main.param_agnostic_general} is indeed tight in all parameters.

\begin{theorem}[Simplified]\label{thm:main.nsgd_lower_bound}
   Under the setting of \Cref{prop:main.param_agnostic_general}, there exists a function $F$ that satisfies \lb, \lsmooth, and an oracle $\nabla f(\cdot, \cdot)$ that satisfies \pBCM\ such that \batchNsgd\ with parameters $\sst \equiv \ssi \li^{-r}$ and $\Bt \equiv \ceil{\max\set{1, B\li^q}}$ requires at least
	\begin{align*}
		\Omega\pare{
			\pare{\frac{\Dz}{\eps}}^{\frac {1+q} {1-r}}
            + \pare{\frac{\sm}{\eps}}^{\frac {1+q} r}
			+ \pare{\frac{\varsym}{\eps}}^{\frac{p \pare{q + 1}}{q\pare{p-1}}}
		}
	\end{align*}
	samples to generate an iterate with $\Expnorm{\nFxt} \leq \eps$.
\end{theorem}

\begin{figure}[t]
	\centering
	\begin{tikzpicture}
	
	\pgfmathsetmacro{\eps}{0.25}        
	\pgfmathsetmacro{\DeltaOne}{2}      
	\pgfmathsetmacro{\L}{1.2}             
	\pgfmathsetmacro{\ssi}{1}           
	\pgfmathsetmacro{\T}{4}           	
	%
	\pgfmathsetmacro{\deltaF}{2*\eps*\ssi - \L*\ssi*\ssi/4}
	%
	\pgfmathsetmacro{\xmin}{-0.2} 
	\pgfmathsetmacro{\ybreaktop}{\DeltaOne - \T * \deltaF}    
	\pgfmathsetmacro{\ymin}{-2.2*\eps}
	\pgfmathsetmacro{\ymax}{\DeltaOne + (\DeltaOne - \ybreaktop)*0.1}
	\pgfmathsetmacro{\ybreakbot}{0.9}     
	
	\pgfplotsset{
		anchor=origin,
		xmin=\xmin, xmax=(\T * \ssi + 0.5),
	}
	
	\begin{axis}[
		name=upper,
		axis x line=middle,
		xlabel={$x$},
		xtick={0, \ssi, 2*\ssi, 3*\ssi},
		xticklabel style={yshift=5mm,xshift=3mm},
		xticklabels={$x_1$, $x_2$, $x_3$, $x_4$},
		axis y line*=left,
		y axis line style={draw=none},
		ymin=\ymin, ymax=\ymax,
		ytick={\DeltaOne,\DeltaOne-\deltaF,\DeltaOne-2*\deltaF,\DeltaOne-3*\deltaF, 2*\eps, 0, -2*\eps},
		yticklabels={$\Dz$, $F(x_2)$, $F(x_3)$, $F(x_4)$, $2\varepsilon$, $0$, $-2\varepsilon$},
		legend style={at={(0.95,0.95)},anchor=north east},
		clip=false,
		legend entries={$F(x)$,$F'(x)$},
		]
		\addlegendimage{blue, very thick}
		\addlegendimage{red, dashed, very thick}
		\draw (\xmin,-2.5*\eps) -- (\xmin,2.5*\eps); 
		\draw[dotted] (\xmin,2.5*\eps) -- (\xmin,\ybreaktop); 
		\draw[-stealth] (\xmin,\ybreaktop) -- (\xmin,\ymax);
 		\pgfplotsinvokeforeach{0,1,...,\T-1}{
 			\pgfmathsetmacro{\domStart}{#1 * \ssi}
 			\pgfmathsetmacro{\domMid}{(#1 + 0.5) * \ssi}
 			\pgfmathsetmacro{\domEnd}{(#1 + 1) * \ssi}
 			\pgfmathsetmacro{\startVal}{\DeltaOne - (#1 * \deltaF)}
			\addplot[blue, very thick, domain=\domStart:\domMid, samples=50] 
			{ \startVal - 2*\eps*(x - \domStart) + 0.5*\L*(x - \domStart)^2 };
			\addplot[blue, very thick, domain=\domMid:\domEnd, samples=50] 
			{ (\startVal - \L*\ssi^2/4 - 0.5*\L*(x-\domStart)^2 + (\L*\ssi - 2*\eps)*(x-\domStart) };
			\addplot[red, very thick, dashed, domain=\domStart:\domMid, samples=50] 
			{ -2*\eps + \L*(x-\domStart) };
			\addplot[red, very thick, dashed, domain=\domMid:\domEnd, samples=50] 
			{ -2*\eps + \L*(\domEnd - x) };
			\edef\temp{
				\noexpand\draw[dotted, thick, opacity=0.5] (axis cs:\domStart, \startVal) -- (axis cs:\domStart, -2*\eps);
				\noexpand\draw[dotted, thick, opacity=0.5] (axis cs:\xmin, \startVal) -- (axis cs:\domStart, \startVal);
			}
			\temp
		}
	\end{axis}

\end{tikzpicture}
	\caption{Plot of the hard instance function and its derivative used in \Cref{thm:main.nsgd_lower_bound}. Dotted lines mark the iterates generated by the algorithm and their function values. $\Dz$ denotes the initial suboptimality $F(x_1)$, and $\varepsilon$ is the target accuracy.}
	\label{fig:lower_bound_function}
\end{figure}

The extended result, which includes the dependence on $\ssi$ and $B$, and its proof can be found in \Cref{sec:app.optimality.stoch}. Its proof is based on two key ideas. First, in the deterministic setting, we construct a hard function that exactly satisfies \lsmooth\ and \lb, penalizing excessively small and large step sizes within a single function, see \Cref{fig:lower_bound_function}. This yields an iteration complexity lower bound of $\Omega\pare{\pare{\nicefrac \Dz \eps}^{1 / \pare{1-r}} + \pare{\nicefrac \sm \eps}^{1 / r}}$. Second, we construct an oracle that points in the opposite direction of the true gradient with maximal probability, while adhering to the \pBCM\ assumption. This oracle leads to a lower bound on the required batchsize, which in turn implies an iteration complexity lower bound of $\Omega\pare{\pare{\nicefrac \varsym \eps}^{\frac{p}{q\pare{p-1}}}}$. Combining these iteration complexity lower bounds with the samples per iteration results in \Cref{thm:main.nsgd_lower_bound}.

\paragraph{Parameter-free convergence.}
When considering the parameters $r = \nicefrac 1 2$ and $q = 1$, \Cref{prop:main.param_agnostic_general} implies the sample complexity  
\begin{align}\label{eq:main.optimal_param_free_result}
	\Oc\pare{\frac{\Dz^4 + \sm^4}{\eps^4} + \pare{\frac \varsym \eps}^{\frac{2p}{p-1}}},
\end{align}
without requiring knowledge of any problem-dependent parameters, including the tail index $p$. It turns out, that this choice of step-size and batch-size parameters is the best parameter-free choice possible, in the sense that \eqref{eq:main.optimal_param_free_result} cannot be \emph{uniformly} improved for all $p \in (1,2]$. More precisely, \eqref{eq:main.optimal_param_free_result} is tight for all $p \in (1,2]$, as can be seen by plugging $r = \nicefrac 1 2$ and $q = 1$ into \Cref{thm:main.nsgd_lower_bound}. Furthermore, while a different choice of $r$ and $q$ may improve the sample complexity for \emph{some} $p$, the complexity would get strictly worse for $p = 2$. That is any other choice of $(r,q) \neq (\nicefrac 1 2, 1)$ implies a sample complexity lower bound of $\omega\pare{\eps^{-4}}$, which is strictly worse than the $\Oc\pare{\eps^{-4}}$ we get from \eqref{eq:main.optimal_param_free_result} for $p=2$.

\paragraph{Optimal sample complexity with tuning.}
For algorithms with knowledge of problem parameters, \citet{zhang2020adaptive} provide a sample complexity lower bound for our setting of
\begin{align}\label{eq:pbcm_lower_bound}
	\Omega\pare{\frac{\Dz \sm}{\eps^2} + \frac{\Dz \sm}{\eps^2} \pare{\frac{\varsym}{\eps}}^{\frac p {p-1}}}.
\end{align}
To the best of our knowledge, there are no upper bounds exactly matching this lower bound, leaving the tightness of \eqref{eq:pbcm_lower_bound} an open question. The following result closes this question, by improving the sample complexity of \eqref{eq:main.optimal_param_free_result} --- when given access to problem-parameters --- to tightly match the lower bound in all parameters.

\begin{corollary}[Optimal Sample Complexity]\label{cor:main.optimal_complexity}
	Assume \lb, \lsmooth\ and \pBCM\ with $p \in (1, 2]$. Then the iterates generated by \batchNsgd\ with parameters $\sst \equiv \sqrt{\nicefrac{\Dz}{\sm \li}}$ and $\Bt \equiv \ceil*{\max\set{1,\pare{\frac{\varsym^2 \li}{\Dz \sm}}^{\frac{p}{2p-2}}}}$ satisfy
	\begin{align*}
		\frac 1 \li \iSum \Exp{\nnFxt} \leq 6 \frac{\sqrt{\Dz \sm}}{\sqrt\li}.
	\end{align*}
	In particular the sample complexity is bounded by $\Oc \pare{\frac{\Dz \sm}{\eps^2} + \frac{\Dz \sm}{\eps^2}\pare{\frac \varsym \eps}^{\frac p {p-1}}}$.
\end{corollary}

For comparison, \citet{zhang2020adaptive} derived in-expectation convergence for \clippedName with a sample complexity of\footnote{We ignore non-leading terms and simplify the rate in their favour.}
\begin{align*}
	\Oc\pare{
		\frac{\Dz \sm \varsym^{\frac{p^2}{p-1}}}{\eps^2} 
		+ \frac{\pare{\Dz \sm \varsym^p}^{\frac{3p-2}{2p-2}} + \varsym^{\frac{3p-2}{p-1}} }{\eps^\frac{3p-2}{p-1}}
	}
\end{align*}
which is suboptimal in all parameters besides $\eps$.


\subsection{Convergence with High-Probability}
\label{sec:main.hp}

While in-expectation results guarantee small gradient norms given sufficiently many optimization runs, computational constraints often preclude running enough procedures. Therefore, results of the form \emph{with probability at least $1-\delta$, a single optimization run achieves a certain gradient norm}, often called in-probability results, are more desirable. While the Markov inequality can convert in-expectation guarantees to in-probability guarantees, the poor polynomial dependence on $\nicefrac 1 \delta$ renders these results impractical.

Therefore, the gold standard are so called high-probability results with a mild $\polylog\pare{\nicefrac 1 {\delta}}$ dependence. To achieve such results, existing literature relies on either light tail noise assumptions (e.g., \citet{ghadimi2013stochastic,hpSGDM2020Li,subWeibull_HP2020Madden,pmlr-v162-li22q,fatkhullin2024taming}), or the gradient clipping mechanism (e.g., \citet{cutkosky2021high,sadiev2023high,nguyen2023improved}). In contrast, \citet{sadiev2023high} prove that vanilla \sgdName (without clipping) cannot achieve a better $\delta$ dependence than $\Omega\pare{\nicefrac 1 {\sqrt \delta}}$ under heavy-tailed noise.

The following Theorem provides a unified high-probability guarantee for \nsgd\ with a general gradient estimator. The result will imply high-probability convergence for \batchNsgd, and high-probability convergence to a $\varsym$-neighbourhood for vanilla \nsgd.
\newline

\newcommand{\ssmax}{\iterationss \li^{\max}}
\begin{theorem}[High-Probability]\label{thm:main.unified_hp}
	Let $\delta \in (0,1)$. Assume \lb, \lsmooth\ and $\infty > \sigt \coloneqq \Exp[\Fctm]{\norm{\gt - \nFxt}}$, where $\Fctm \coloneqq \sigma\pare{\iterationg \fin, \dots, \gtm}$. Additionally let $\ssmax\coloneqq \max_{\ii \in [\li]}\sst$ and $C_\li \coloneqq \max_{\ii \in [\li]} \sst \sum_{\tau=1}^{\ii-1} \sstau$. Then, with probability at least $1-\delta$, the iterates generated by \nsgd\ satisfy
	\begin{align*}
		\iSum w_t \norm{\nabla_\ii}
		\leq &\ 
		\frac{2\Dz + \sm \iSum \sst^2 + 4 \iSum \sst \sigt}{\tauSum \sstau} \\
		&\ + \frac{12\pare{\ssmax\norm{\nF \pare \xfin} + C_\li \sm} \logood} {\tauSum \sstau},
	\end{align*}
	where $w_t \coloneqq \frac{\sst}{\tauSum \sstau}$ and $\nabla_\ii \coloneqq \nFxt$.
\end{theorem}

The main idea behind the proof is to reduce the statement to lower bounding the expected cosine between $g_t$ and $\nFxt$. Since the cosine is bounded within $[-1, 1]$, we can apply a high-probability concentration inequality on it and obtain the mild $\logood$ dependence. We would like to point out that our proof technique for establishing this high probability result significantly deviates from the existing high probability analysis of methods using gradient clipping. The formal proof can be found in \Cref{sec:app.missing.hp}.

Note that \Cref{thm:main.unified_hp} does not make any unbiasedness or decreasing stepsize assumptions. Furthermore, when comparing this result with its in-expectation counterpart (see \Cref{prop:app.missing.exp.unified_expectation}), the bound can be interpreted as a concentration inequality around the expected value. 
Crucially, compared to \citep{cutkosky2021high}, our algorithm does not require the additional clipping mechanism, effectively reducing the need to tune an additional parameter and aligning theory with practice. 

We next apply \Cref{thm:main.unified_hp} to \batchNsgd\ with optimal parameters. A parameter-free version can be found in \Cref{sec:app.missing.hp.param_free}. To the best of our knowledge, we are the first work to show a high-probability result without requiring strong noise assumptions or clipping.\footnote{Note that some works, e.g., \citep{NonlinearHP2023Armacki}, provide high-probability guarantees for \nsgd\ under different noise assumptions. Specifically, it is important for their technique to assume the existence of a symmetric PDF. }

\begin{corollary}\label{cor:main.optimal_hp}
	Assume \lb, \lsmooth\ and \pBCM\ with $p \in (1,2]$. Then the iterates generated by \batchNsgd\ with parameters $\sst \equiv \sqrt{\frac{\Dz}{\sm\li}}$ and $B_t \equiv \ceil*{\max\set{1, \pare{\frac{\varsym^2\li}{\Dz \sm}}^{\frac{p}{2p-2}} }}$ satisfy
	\begin{align*}
		\frac 1 \li \iSum \nnFxt
		&\leq \pare{11 + 30\logood}\frac{\sqrt{\Dz \sm}}{\sqrt\li}
	\end{align*}
	with probability at least $1-\delta$. This corresponds to a $\Oct \pare{\frac{\Dz \sm}{\eps^2} + \frac{\Dz \sm}{\eps^2} \pare{\frac{\varsym}{\eps}}^{\frac{p}{p-1}}}$ sample complexity.
\end{corollary}

This result is optimal in $\Delta_1, \sm, \varsym$ and $\eps$. We are not aware of any lower bounds specifying the optimal $\delta$-dependence. For comparison, \citet{nguyen2023improved} derived high-probability convergence of \clippedName with a sample complexity of
\begin{align*}
	\Oct\pare{
		\frac{\pare{\Dz \sm}^{\frac{3p-2}{4p-4}}}{\eps^{\frac{6p-4}{2p-1}}}
		+ \frac{
			\pare{\frac{\varsym^{2p}}{\pare{\Dz \sm}^{2-p}}}^{\frac{3p-2}{p-1}} 
			+ \pare{\Dz \sm \varsym^2}^{\frac{3p-2}{4p-4}}
		}{\eps^{\frac{3p-2}{p-1}}}
	},
\end{align*}
which is suboptimal in all parameters besides $\eps$ in the stochastic case. In the deterministic case, even the dependence on $\eps$ appears suboptimal. In contrast, our result is noise adaptive in the sense that, for $\varsym = 0$, the optimal deterministic iteration complexity is obtained. In particular, this result closes open questions posed by \citet{UnboundedClippedNSGDM2023Liu}.

While \Cref{thm:main.unified_hp} can be applied to show high-probability convergence of vanilla \nsgd\ to a $\varsym$-neighbourhood, technical difficulties prevent us from extending it to \nsgd\ with momentum. We investigate this empirically in \Cref{sec:experiments}, and describe the technical challenges in \Cref{sec:app.nsgdm.difficulties}.

\subsection{Can we Improve the Convergence Measure of Normalized \texorpdfstring{\sgdName}{SGD}?}
\label{sec:main.improve_measure}

One may observe that the convergence of \nsgd\ above is stated in terms of the average gradient norm, which is different from the average of \textit{squared} gradient norm that is commonly used in non-convex optimization. By Jensen's inequality it is straightforward to see that 
\begin{equation}\label{eq:jensen_L1_L2}
\sqrt{ \frac{1}{T}\sum_{t=1}^T \sqnorm{\nabla F( x_t)} } \geq \frac{1}{T}\sum_{t=1}^T \norm{\nabla F( x_t)} ,
\end{equation}
This raises a natural question whether this different convergence measure is a limitation of our analysis or an intrinsic property of the algorithm. The following result shows that this is indeed an intrinsic property of the algorithm, by providing a lower bound on the second moment of the gradient norm.

\begin{theorem}\label{thm:main.lb}
	There exists an $L$-smooth function $F\colon \R \rightarrow \R$ such that if \batchNsgd\ with parameters as in \Cref{cor:main.optimal_complexity} is run from any initial point $\xfin > 0 $ for $T \geq 18$ iterations, then we have 
	\begin{align*}
		\sqrt{ \frac{1}{\li}\iSum \sqnorm{\nabla F( x_t)} } 
        &\geq 
		\frac 2 3 \frac{ \sqrt{ \Dz \sm}}{\sqrt \li } \cdot \li^{1/4} ,
	    \quad \text{while} \\
		\frac{1}{T}\sum_{t=1}^T \norm{\nabla F( x_t)}  
        &\leq 6  \frac{\sqrt{\Dz \sm} }{  \sqrt \li }  \quad \text{(by \Cref{cor:main.optimal_complexity}).}
	\end{align*}
\end{theorem}

The above result implies that even if we select the optimal step-size parameter (to minimize the upper bound on $\frac{1}{T}\sum_{t=1}^T \norm{\nabla F( x_t)}$), \batchNsgd\ does not achieve optimal rates in terms of the stronger measure $\sqrt{ \frac{1}{T}\sum_{t=1}^T \sqnorm{\nabla F( x_t)} } $. Specifically, the convergence rate for the latter measure is worse at least by a factor of $\Theta(T^{\nicefrac 1 4})$. Our \Cref{thm:app.optimality.deterministic_polynomial_ngd_optimiality} implies that the step-size $\sst = \sqrt{\frac{\Dz}{\sm\li}}$ is essentially the only order of step-size to attain the optimal convergence rate in terms of $\frac{1}{T}\sum_{t=1}^T \norm{\nabla F( x_t)}$ (up to a numerical constant). Combined with inequality \eqref{eq:jensen_L1_L2} and \Cref{thm:main.lb}, it implies that there is no predefined constant step-size for \nsgd\ that can guarantee optimal convergence when the rate is measured by $\sqrt{ \frac{1}{T}\sum_{t=1}^T \sqnorm{\nabla F( x_t)} } $.

\section{EXPERIMENTS}
\label{sec:experiments}
In this section, we present experiments designed to empirically motivate and validate the theoretical findings of this paper. Since heavy tails have prominently been observed in language modelling tasks \citep{zhang2020adaptive}, our experiments target this task.

\paragraph{Experimental Setup.} We conduct training on the Penn Treebank (PTB) \citep{PTB1993Marcus} and WikiText-2 \citep{Wikitext2017Merity} datasets using the AWD-LSTM architecture \citep{LSTM2018Merity}. Hyperparameters of the model and batchsizes were chosen according to \citep{LSTM2018Merity}. To observe the exact optimization behaviour of algorithms, the averaging mechanism of the model was disabled. Additional licensing and compute information can be found in \Cref{sec:app.add_experiments}.

In order to examine the behaviour of \clippedSGD and compare it to \nsgd, we tuned their respective parameters using a course grid search in a 50 epoch training. For \nsgd\ we considered the stepsizes $\sst = \ssi \ii^{-r}$ and tuned $\ssi$ and $r$. For \clippedName we considered the same stepsizes and additionally tuned the clipping threshold $\clipThresh$. The parameters resulting on the above described tuning scheme on the PTB dataset were  $\pare{\ssi, r ,\clipThresh} = \pare{50,0.1,0.25}$ for \clippedName and $\pare{\ssi, r} = (50, 0.25)$ for \nsgd. It should be noted that the observed optimal clipping threshold $\clipThresh = 0.25$ is in line with the previous empirical work by \citet{LSTM2018Merity} that introduced the AWD-LSTM. The resulting parameters on the WikiText-2 dataset were $\pare{\ssi, r ,\clipThresh} = \pare{30,0,0.15}$ for \clippedName and $\pare{\ssi, r} = (15, 0.1)$ for \nsgd.
The final training was then carried out for $300$ epochs on the seeds $0, 1970, 2000, 2024$ and $2112$.

\paragraph{Motivation and Validation.}\Cref{fig:clipped_behaviour} shows the behaviour of \clippedName and \nsgd\ with their corresponding tuned parameters on both datasets. Dashed lines represent the proportion of stochastic gradients that got clipped by \clippedName per epoch. We want to discuss two observations. First, perhaps surprisingly, we observe on both datasets that the percentage of events when gradients are clipped increases for \clippedName, contradicting theoretical insights as discussed in \Cref{subsec:drawbacks_clipoing}. Interestingly, \clippedName eventually clips at every iteration, becoming equivalent to \nsgd\ after a certain number of epochs. Second, it can be noted that both algorithms perform similarly when measured with their corresponding training loss, depicted with solid lines, despite \nsgd\ having one less parameter and hence requiring substantially less time to tune.

\begin{figure*}[!ht]
	\centering
	\hfill
	\begin{subfigure}[c]{0.45\linewidth}
		\includegraphics[width=\linewidth]{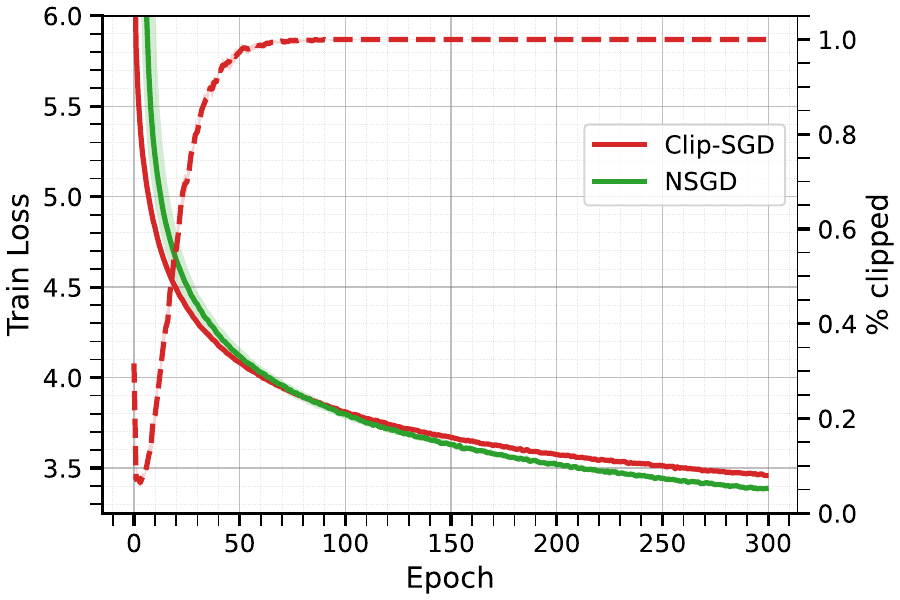}
		\caption{PTB}
		\label{fig:ptb.clipping_percentage}
	\end{subfigure}
	\hfill
	\begin{subfigure}[c]{0.45\linewidth}
		\includegraphics[width=\linewidth]{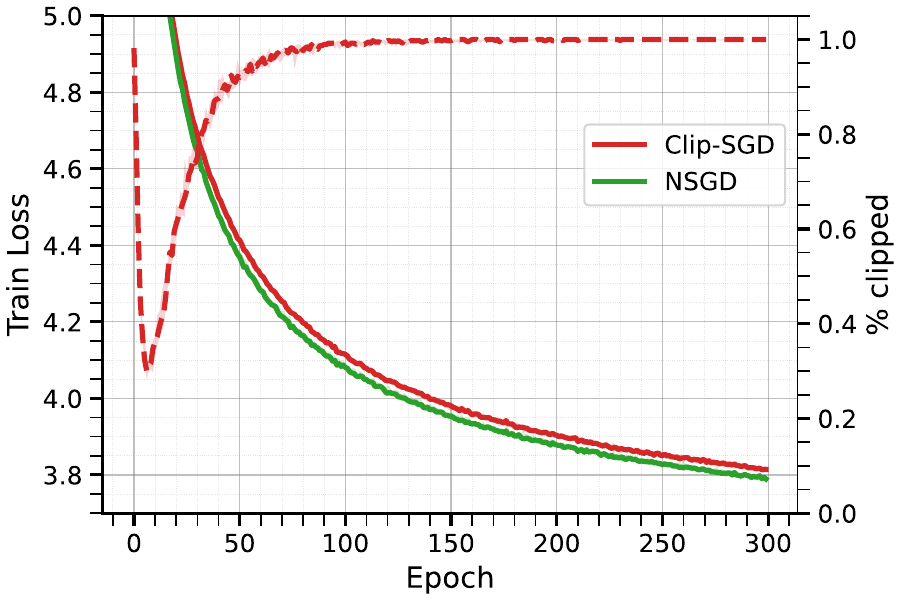}
		\caption{WikiText-2}
		\label{fig:wikitext.clipping_percentage}
	\end{subfigure}
	\hfill
	\caption{All plots consider \clippedName\ and \nsgd\ with tuned parameters. Solid lines represent the training loss and correspond to the left y-axis. The dashed line correspond to the right y-axis, and represent the percentage of clipped gradients by \clippedSGD\ in an epoch. Shaded areas represent the minimal and maximal value within 5 seeds, the line the median.}
	\label{fig:clipped_behaviour}
\end{figure*}

\begin{figure*}[!ht]
    \centering
    \includegraphics[width=0.95\linewidth]{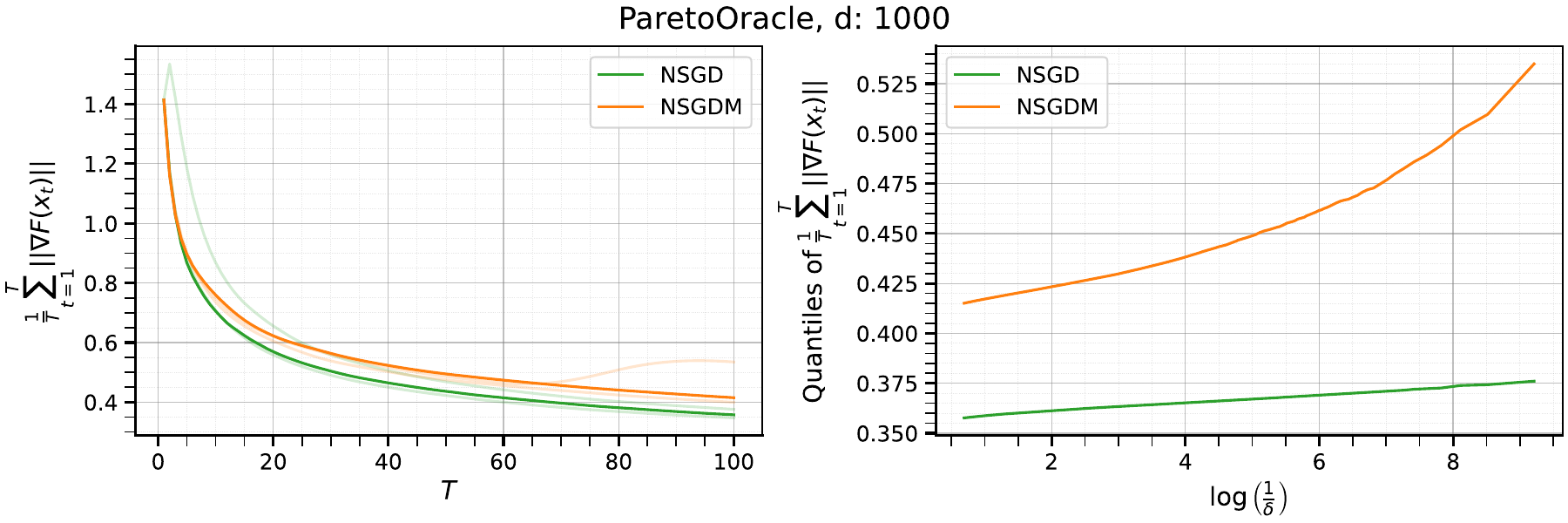}
	\caption{Verifying high probability convergence of \nsgd\ and \nsgdm. We use $f(x, \xi) = \frac{1}{2}\sqnorm{x}_2 + \langle x, \xi \rangle$, where $\xi$ is a random vector with i.i.d.\,components drawn from a symmetrized Pareto distribution with tail index $p = 2.5$. For \nsgdm, the quantiles of average gradient norm (left plot) exhibit a super-linear dependence on $\log(1/\delta)$ indicating the lack of high probability convergence.}
\label{fig:d1000k100000var1_pareto25_nsgdm}
\end{figure*}

\paragraph{High probability convergence.} In this set of experiments, we verify high probability convergence of \nsgd, and \nsgdm\ on a strongly convex quadratic function under heavy tailed noise. We run each algorithm for $k = 10^5$ times with default parameters: $\eta_t = 1/\sqrt{t}$ for \nsgd\ and $\eta_t = \sqrt{\alpha_t / t}$, $\alpha_t = 1/\sqrt{t}$ for \nsgdm\, see \eqref{eq:app.nsgdm} in \Cref{sec:app.nsgdm}. \Cref{fig:d1000k100000var1_pareto25_nsgdm} (left) visualizes the convergence behaviour by selecting the median along with $\delta$ and $1-\delta$ quantiles of the algorithm runs based on the average gradient norm at $T = 100$, where $\delta := 1-10^{-4}$. We observe that the $\delta$ quantile run of \nsgdm\ deviates significantly from the median compared to that of \nsgd\ from its own median. \Cref{fig:d1000k100000var1_pareto25_nsgdm} (right) plots the average gradient norm at $T = 100$ corresponding to different values of quantiles $\delta$. \nsgdm\ shows a super-linear dependence on $\log(1/\delta)$, indicating that it may not achieve high probability convergence as effectively as \nsgd\, even in the presence of BV noise. This suggests that extending high probability bounds to \nsgdm\, is more challenging due to momentum's impact on the dynamics. We refer to \Cref{sec:app.nsgdm.difficulties} for a more detailed discussion on theoretical challenges of showing high-probability bound for \nsgdm, and to \Cref{sec:appendix:verify_hp} for additional experiments including comparison with light tail noise and \sgd. 

\section{CONCLUSION}
\label{sec:conclusion}
This work analyzes Normalized \sgd\ under heavy-tailed noise. Our theoretical analysis reveals several interesting insights. First, we extend our understanding of high-probability convergence under heavy tailed noise, providing the first such guarantee with an algorithm that does not require gradient clipping. Second, we tightly characterize the optimal sample complexity in all parameters under the \pBCM\ assumption. Lastly, our results for parameter-free \nsgd\ suggest the robustness of the algorithm to misspecification of its parameters. Additionally, our algorithm-specific lower bounds in \Cref{thm:main.nsgd_lower_bound} and \Cref{thm:main.lb} allow additional insights into the behavior of \nsgd.

Several open questions arise from this work for future research. For instance, it remains unclear whether our high-probability result can be extended to \nsgd\ with momentum or variance-reduced gradient estimators. More importantly, it remains open whether the sample complexity that is optimal for parameter-dependent algorithms \eqref{eq:pbcm_lower_bound} can be achieved by any algorithm without knowledge of problem parameters.

\subsubsection*{Acknowledgements}
The work is supported by ETH AI Center Doctoral Fellowship and Swiss National Science Foundation (SNSF) Project Funding No. 200021-207343.

\bibliographystyle{abbrvnat}
\bibliography{./files/biblio}


\clearpage              
\onecolumn              
\appendix               
\tableofcontents        
\clearpage              
\clearpage

\section{SUMMARY OF TECHNICAL CONTRIBUTIONS}
\label{sec:app.technical_contributions}
In the following we want to summarise the main technical contributions of our work to simplify using the developed technique in other proofs. We group the contributions by topic, so other authors can focus on the area they are interested in.

\paragraph{In-Expectation Upper-Bounds under \pBCM.} With \Cref{lem:app.von_Bahr_and_Essen} in combination with \Cref{lem:app.technical.cond_exp} we provide rigorous tools to control the expected deviation of momentum and mini-batch gradient estimators from the true gradient. Examples of applications can be found in \eqref{eq:app.missing.exp.param_agnostic_general.pointwise}-\eqref{eq:app.missing.exp.param_agnostic_general.cond_exp_dev} for the mini-batch gradient estimator and in \eqref{eq:app.nsgdm.application_vBE.start}-\eqref{eq:app.nsgdm.application_vBE.end} for the momentum estimator.

\paragraph{In-Expectation Lower-Bounds.} We provide a hard function and oracle that is able to tightly characterize the parameter-dependence of \algname{NSGD}\ in \Cref{lem:app.optimality.deterministic_arbitrary_ngd_optimiality} and \Cref{thm:app.optimality.stochastic}. Very similar constructions can be used to derive a similar result for \algname{SGD} and possibly other algorithms.

\paragraph{High-Probability Upper-Bounds under \pBCM.} We lay out a different convergence analysis of \nsgd\ that hinges on controlling $\phit \coloneqq \frac{\gt^\top \nFxt}{\norm{\gt} \nnFxt}$ instead of $\norm{\gt - \nFxt}$, which is used in previous analyses. This allows to prove high-probability guarantees due to the boundedness of $\phit$. In contrast, \citet{cutkosky2021high} require an additionally gradient clipping to guarantee concentration of $\norm{\gt - \nFxt}$ with high probability.

\paragraph{Difference between Convergence Measures.} We construct a function which explicitly showcases that the convergence rate of \nsgd\ deteriorates when different convergence measures are used. Similar constructions could potentially be applied to other algorithms and convergence measures.
\clearpage

\section{TECHNICAL RESULTS}
\label{sec:app.technical_results}
This section contains various technical results required for our analysis. We start with two lemmas that arise due to the normalization in \nsgd. Slightly different formulations of these were used by \citet{NSGD2021Zhao}.

\begin{lemma}\label{lem:app.remove_normalisation}
	For all $a,b \in \R^d$ with $b \neq 0$ we have
	\begin{align*}
		\frac{a^\top b}{\norm b} \geq \norm a - 2 \norm{a-b}.
	\end{align*}
\end{lemma}
\begin{proof}
	We calculate 
	\begin{align*}
		\frac{a^\top b}{\norm b} 
		= \frac{\pare{a-b}^\top b}{\norm b} + \norm b
		\geq -\norm{a-b} + \norm b
		\geq \norm a - 2 \norm{a-b},
	\end{align*}
	where we used Cauchy-Schwarz in the first, and $\norm{a} \leq \norm{a-b} + \norm{b}$ in the second inequality.
\end{proof}

\begin{lemma}[Expected Angle Bound]\label{lem:app.expected_angle}
	Let $\pare{\Omega, \mathfrak{A}, \P}$ be a probability space and $X \colon \Omega \to \R^d$ a random vector. Furthermore let $\mu \in \R^d \setminus \set{0}, \sigma \coloneqq \Exp{\norm{X - \mu}}$ and suppose that $X \neq 0$ almost surely. Then it holds that
	\begin{align*}
		\Exp{\frac{\mu^\top X}{\norm{\mu}{\norm{X}}}} \geq 1 - 2 \frac{\sigma}{\norm{\mu}}.
	\end{align*}
\end{lemma}

\begin{proof}
	We apply \Cref{lem:app.remove_normalisation} with $a \gets \mu$ and $b \gets X$ to derive
	\begin{align*}
		\frac{\mu^\top X}{\norm X} \geq \norm \mu - 2 \norm{\mu - X}.
	\end{align*}
	Dividing both sides by $\norm \mu$ and taking expectations yields the claim.
\end{proof}

The next lemma shows that $\ii$-dependent parameters have the \emph{same} (up to constants) behavior as constant, $\li$-dependent, parameters in \nsgd.

\begin{lemma}[{see \citet[Lemma 10]{NSGDM_LzLo2023Huebler}}]\label{lem:app.technical.dec_eq_const}
	Let $q \in (0,1), r \in [0,1]$ and $\li \in \N_{\geq 2}$. Then
	\begin{align*}
		\sum_{\ii = 1}^\li \ii^{-r} \prod_{\tau = \ii + 1}^\li \pare{1-\tau^{-q}}
		\leq 2\exp\pare{\frac 1 {1-q}}\pare{\li+1}^{q-r}.
	\end{align*}
\end{lemma}

To control the error of momentum or mini-batch gradient estimator, we use a von Bahr and Esseen type inequality stated in \Cref{lem:app.von_Bahr_and_Essen}. This result was initially proved by \citet{von1965inequalities} for $d=1$. Later, it was extended to Banach-Spaces \citep[Proposition 2.4]{Pisier1975} and optimal constants were derived \citep{Cox_1982}. An alternative proof was rediscovered by \citet{cherapanamjeri2022optimal} for the one dimensional i.i.d.\ case and extended to higher dimensions by \citet{kornilov2024accelerated}. The extension in the latter work has some inaccuracies, which we fix below. 

\begin{lemma}\label{lem:app.von_Bahr_and_Essen}
	Let $p \in [1,2]$, and $X_1, \ldots, X_n \in \R^d$ be a martingale difference sequence (MDS), i.e., $\Exp[X_{j-1}, \ldots, X_1]{X_{j}} = 0$ a.s. for all $j = 1, \ldots, n$ satisfying 
	\begin{align*}
	\Exp{\norm{X_j}^p } < \infty \qquad \text{for all } j = 1, \ldots, n. 
	\end{align*}
	Define $S_n :=  \sum_{j=1}^n X_j$, then 
	\begin{align*}
		\Exp{\norm{S_n}^p} \leq 2 \sum_{j=1}^n \Exp{\norm{X_j}^p}.
	\end{align*}
\end{lemma}

\begin{proof}
 	The claim is true for $d=1$, see \citep[Equation (4)]{von1965inequalities}. Namely, let 
 	$y_1, \ldots, y_n \in \R$ be a MDS, i.e., $\Exp[y_{j-1}, \ldots, y_1]{y_{j}} = 0$ a.s. for all $j = 1, \ldots, n$, satisfying 
 	\begin{align}\label{eq:pBCM_d1}
 		\Exp{|y_j|^p } < \infty \qquad \text{for all } j = 1, \ldots, n. 
 	\end{align}
	Then 
 	\begin{align}\label{eq:vBE_d1}
 		\Exp{\left|\sum_{j=1}^{n} y_j\right|^p} \leq 2 \sum_{j=1}^n \Exp{|y_j|^p}.
 	\end{align}
 	Following \citep{kornilov2024accelerated}, define $g \sim \mathcal N(0, I)$ and $y_j := g^{\top} X_j$, where $g$ is independent of $X_j$. We need to verify that $y_1, \ldots, y_n$ defined this way is indeed a MDS. Define the sigma algebra $\mathcal H_1 := \sigma(y_{j-1}, \ldots, y_1)$ and $\mathcal H_2 := \sigma(X_{j-1}, \ldots, X_1, g)$.  Observe that $\mathcal H_1 \subset \mathcal H_2$ and by the tower property
 	\begin{align*}
	 	\Exp{g^{\top} X_j |\mathcal H_1} =  \Exp{ \Exp{ g^{\top}   X_j | \mathcal H_2 } | \mathcal H_1 } = 	\Exp{  g^{\top}  \Exp{ X_j | \mathcal H_2 } | \mathcal H_1 } = 0 , 
 	\end{align*}
 	
 	where the second equality holds because $g$ is $\mathcal H_2$-measurable and the last equality holds by independence of $X_j$ and $g$, and the assumption that $\Exp[X_{j-1}, \ldots, X_1]{X_{j}} = 0$ a.s. Next, we need to verify that $\Exp{|y_j|^p} < \infty$. We know that $g^{\top} a \sim \mathcal N(0, \sqnorm{a})$ for any vector $a\in \R^d$. Therefore, 
 $
 		\Exp{|y_j|^p | X_j} = \Exp{| g^{\top} X_j |^p | X_j} = C(p) \norm{X_j}^p , 
 $
 	with $C(p) := 2^{\nicefrac{p}{2}} \frac{\Gamma\rb{\frac{p+1}{2}}}{\sqrt{\pi}}$, where we used the $p$-th absolute moment of normal distribution applied to a random variable $g^{\top} X_j$ given $X_j$. Taking full expectation, we get 
\begin{equation}\label{eq:exp_lp}
   \Exp{|y_j|^p} = C(p) \, \Exp{\norm{X_j}^p } < \infty. 
  \end{equation}
	We have verified that the sequence $y_1, \ldots, y_n$ is a MDS and property \eqref{eq:pBCM_d1} holds, thus we are ready to apply \eqref{eq:vBE_d1}. Using the $p$-th moment of normal distribution applied to $g^{\top} S_n$ given $S_n$, we have 
	\begin{align*}
		C(p) \, \Exp{\norm{S_n}^p} =  \Exp{\Exp{| g^{\top} S_n |^p | S_n} } = \Exp{\left| \sum_{j=1}^n y_j \right|^p } \leq 2 \sum_{j=1}^n \Exp{|y_j|^p} ,
	\end{align*}
	where in the last step we used \eqref{eq:vBE_d1}. It remains to use \eqref{eq:exp_lp} to conclude the proof. 

\end{proof}

We use the following martingale concentration inequality for our high-probability guarantees, see e.g., \citep[Lemma 1]{hpSGDM2020Li}. 

\begin{lemma}\label{lem:app.technical.conc_ineq_MDS}
	Let $\pare{\Fct}_{\ii \in \N}$ be a Filtration and $\pare{D_t}_\tinN$ a Martingale Difference Sequence with respect to $\pare{\Fct}_\tinN$. Furthermore, for each $\ii \in \Ngeq$, let $G_t$ be $\Fctm$-measurable and assume that $\Exp[\Fctm]{\exp\pare{\frac{D_t^2}{G_t^2}}} \leq e$. Then, for all $T \in \N$,
	\begin{align*}
		\fas \lambda > 0, \delta \in \pare{0,1}\colon \P \pare{\iSum D_t \leq \frac 3 4 \lambda \iSum G_t^2 + \frac 1 \lambda \log \pare{\frac 1 \delta}} \geq 1 - \delta.
	\end{align*}
\end{lemma}

In order to apply \Cref{lem:app.von_Bahr_and_Essen}, we require the following technical Lemmas on conditional expectations. The first Lemma formalises the intuition that \emph{conditioning on $X$} corresponds to \emph{fixing $X$}.

\begin{lemma}[{c.f.\ \citet[Example 4.1.7]{probability2019Durrett}}] \label{lem:app.technical.cond_exp}
	Let $\pare{\Omega, \Fc, \P}$ be a probability space and $X, Y$ be independent random variables mapping to measurable spaces $\pare{E_1, \Sigma_1}$ and $\pare{E_2, \Sigma_2}$ respectively. Furthermore let $h \colon E_1 \times E_2 \to \R^d$ be a (Lebesgue-)measurable function with $\Exp{\norm{h\pare{X,Y}}} < \infty$. Then
	\begin{align*}
		\Exp[X]{h\pare{X,Y}} \stackrel{\text{a.s.}}{=} g(X), \qquad \text{where} \qquad g(x) \coloneqq \Exp{h\pare{x, Y}}.
	\end{align*}
\end{lemma}
\begin{proof}
	\newcommand{\ind}{1}
	First note that, by Fubini's Theorem, $g$ is $\Sigma_1 / \mathcal{B}^d$ measurable and hence $g(X)$ is $\sigma\pare{X} / \mathcal{B}^d$ measurable. Therefore it suffices to show that
	\begin{align*}
		\Exp{h\pare{X,Y} \ind_A} = \Exp{g(X) \ind_A}
	\end{align*}
	for all $A \in \sigma\pare X$. First note that, by definition of $\sigma(X) = \set{X^{-1} \pare{C}\colon C \in \Sigma_1}$, there exists $B \in \Sigma_1$ with $A = X^{-1}(B)$. Next, by independence of $X$ and $Y$, their joint induced measure is a product measure $\mu \otimes \nu$ on $E_1 \times E_2$. Combining, we get
	\begin{align*}
		\Exp{h\pare{X,Y}\ind_A}
		= \int_{A} h\pare{X\pare{\omega},Y\pare{\omega}} d\P\pare{\omega}
		= \int_{E_1 \times E_2} h\pare{x,y} \ind_B\pare{x} d \pare{\mu \otimes \nu} \pare{x,y}.
	\end{align*}
	By our assumption $\Expnorm{h\pare{X,Y}} < \infty$ we know that $h$ is $\mu \otimes \nu$ integrable and Fubini's Theorem hence yields
	\begin{align*}
		\int_{E_1 \times E_2} h\pare{x,y} \ind_B\pare{x} d \pare{\mu \otimes \nu} \pare{x,y}
		= \int_{E_1} \int_{E_2} h\pare{x,y} d \nu\pare{y} \ind_B\pare{x} d \mu \pare x
		= \int_{E_1} g \pare x \ind_B\pare x d \mu \pare x
		= \Exp{g \pare X \ind_A}.
	\end{align*}
	This completes the proof.
\end{proof}

The second Lemma allows us to formally \emph{remove unnecessary conditioning}.

\begin{lemma}\label{lem:app.technical.remove_unnecessary_cond}
	Let $\pare{\Omega, \Fc, \P}$ be a probability space and $X, Y$ be random variables mapping to measurable spaces $(E_1, \Sigma_1)$ and $(E_2, \Sigma_2)$ respectively. Furthermore let $\cH \subseteq \Fc$ be a sigma algebra such that $\sigma(X) \subseteq \cH$ and assume $\sigma(Y)$ is independent of $\cH$. Finally let $h \colon E_1 \times E_2 \to \R^d$ be a (Lebesgue-)measurable function with $\Exp{\norm{h(X,Y)}} < \infty$. Then
	\begin{align*}
		\Exp[\cH]{h(X,Y)} \overset{\text{a.s.}}{=} \Exp[X]{h(X,Y)}.
	\end{align*}
\end{lemma}

\begin{proof}
	\newcommand{\ind}{\mathbf{1}}
	First note that, due to $\sigma(X) \subseteq \cH$, $\Exp[X]{h(X,Y)}$ is $\cH$-measurable and hence it suffices to show
	\begin{align*}
		\fas A \in \cH \colon \Exp{\Exp[X]{h(X,Y)} \cdot \ind_A} = \Exp{h(X,Y) \cdot \ind_A}.
	\end{align*}
	Note that, by independence of $\sigma(Y)$ and $\cH \supseteq \sigma(X)$, $X$ and $Y$ are independent and their joint induced measure is hence a product measure $\mu \otimes \nu$ on $E_1 \times E_2$. Therefore let $A \in \cH$ and define $g(x) \coloneqq \Exp{h(x, Y)}$. Hence, by \Cref{lem:app.technical.cond_exp}, we have $\Exp{\Exp[X]{h(X,Y)} \cdot \ind_A} = \Exp{g(X) \cdot \ind_A}$. Next let $\nu$ be the push-forward measure of $Y$ on $E_2$ and $\mu$ the push-forward measure of $(X, \ind_A)$ on $E_1 \times \{0,1\}$. Then
	\begin{align*}
		\Exp{g(X) \cdot \ind_A}
		&= \int_{E_1 \times \{0,1\}} \pare{\int_{E_2} h(x,y) \, d\nu(y)} \chi \, d\mu(x, \chi)
		= \int_{E_2 \times E_1 \times \{0,1\}} h(x,y) \chi \, d(\nu \otimes \mu)(y,x,\chi),
	\end{align*}
	where we used Fubini in the last equality. Finally, by independence of $Y$ and $(X, \ind_A)$, we have that $\nu \otimes \mu = \gamma$, where $\gamma$ is the joint push-forward measure of $(Y, X, \ind_A)$. Combining, we hence have
	\begin{align*}
		\Exp{\Exp[X]{h(X,Y)} \cdot \ind_A}
		&= \int_{E_2 \times E_1 \times \{0,1\}} h(x,y) \chi \, d\gamma(y,x,\chi) 
		= \int_\Omega h(X(\omega), Y(\omega)) \ind_A(\omega) \, d\mathbb{P}(\omega) 
		= \Exp{h(X,Y) \cdot \ind_A}.
	\end{align*}
	This completes the proof.
\end{proof}
\clearpage

\section{UPPER-BOUNDS FOR NSGD}
\label{sec:app.missing}
{
\newcommand{\ssl}{\bar{\ssi}}
\newcommand{\Bl}{\bar{B}}
\newcommand{\xitGen}[1]{\xit^{\pare{#1}}}
\newcommand{\xitj}{\xitGen{j}}
\newcommand{\Hctj}{\cH_\ii^{(j)}}
\newcommand{\Hctjt}{\widetilde{\cH}_\ii^{(j)}}

This section contains the proofs that are missing in the main part of the paper. Throughout this section we denote the iterates generated by \nsgd\ with $\pare{\xt}_{\tinN}$. Furthermore, we denote the natural filtration of the gradient estimates by $\Fct \coloneqq \sigma \pare{\iterationg \fin, \dots, \gt}$.

We start by deriving a descent lemma. While such descent lemmas are well studied for \nsgdName in the literature --- to the best of our knowledge --- none highlight the importance of the cosine between $\gt$ and $\nFxt$. As this term will play a crucial role in our high-probability result, we will provide our version of the descent lemma and its proof below.

\begin{lemma}[Descent Lemma]\label{lem:app.descent_lemma}
	Assume \lb\ and \lsmooth. Furthermore let
	\begin{align*}
		\phit \coloneqq \frac{\nFxt^\top \gt}{\nnFxt \norm{\gt}}
	\end{align*}
	denote the cosine between $\gt$ and $\nFxt$. Then the iterates of \nsgd\ satisfy
	\begin{align*}
		\iSum \sst \phit \nnFxt \leq \Dz + \frac \sm 2 \iSum \sst^2.
	\end{align*}
\end{lemma} 

\begin{proof}
	By the definition of $\xtp$, \lsmooth\ implies
	\begin{align*}
		\obj\atxtp - \obj \atxt
		\leq - \sst \nFxt^\top \frac{\gt}{\norm{\gt}} + \frac \sm 2 \sst^2
		= - \sst \frac{\nFxt^\top \gt}{\nnFxt \norm{\gt}} \nnFxt + \frac \sm 2 \sst^2.
	\end{align*}
	Summing up over $t \in [\li]$ and telescoping now yields
	\begin{align*}
		\obj^* - \obj\pare{\xfin} \leq \obj \pare{\iterationx {\li + 1}} - \obj\pare{\xfin}
		&\leq -\iSum \sst \phit \nnFxt + \frac \sm 2 \iSum \sst^2,
	\end{align*}
	where we used \lb\ in the first inequality. This completes the proof.
\end{proof}

Thus, if we could guarantee that the angle between the gradient oracle and true gradient remains bounded away from zero, we would be done. Since this can however, even in expectation, not be guaranteed, we need a more detailed analysis to prove our results.


\subsection{In-Expectation Upper-Bounds}
\label{sec:app.missing.exp}

To prove our in-expectation results, we start with a unified analysis for normalized algorithms. This result does not specify the exact gradient estimator, allowing to incorporate different gradient estimators and noise assumptions afterward. This result was first derived by \citet{MomentumImprovesNSGD_Cutkosky_2020} in a slightly different formulation. 

\begin{proposition}[{c.f.\ \citet[Lemma 2]{MomentumImprovesNSGD_Cutkosky_2020}}] \label{prop:app.missing.exp.unified_expectation}
	Assume \lb, \lsmooth\ and $\infty > \sigt \coloneqq \Exp{\norm{\gt - \nFxt}}$. Then the iterates $\pare{\xt}_{\ii \in \Ngeq}$ generated by \nsgd\ satisfy
	\begin{align*}
		\iSum \frac{\sst}{\tauSum \sstau} \Exp{\norm{\nFxt}}
		&\leq \frac{\Dz + \frac \sm 2 \iSum \sst^2 + 2 \iSum \sst \sigt}{\tauSum \sstau}.
	\end{align*}
\end{proposition}

Note that for constant parameters $\sst \equiv \ssi$ and $\varsym_t \equiv \varsym$ this result reduces to
\begin{align}\label{eq:app.missing.exp.unified_exp.constant}
	\frac 1 \li \iSum \Exp{\norm{\nFxt}}
	&\leq \frac{\Dz}{\ssi \li} + \frac{\ssi \sm}{2} + 2 \varsym.
\end{align}
We provide a slightly different proof of Proposition \ref{prop:app.missing.exp.unified_expectation} when compared to \citep{MomentumImprovesNSGD_Cutkosky_2020} below.

\begin{proof}
	Let $\phit \coloneqq \frac{\nFxt^\top \gt}{\nnFxt \norm{\gt}}$ denote the cosine between $\nFxt$ and $\gt$. Then, by \Cref{lem:app.descent_lemma}, we have
	\begin{align}\label{eq:app.unified_expectation.descent_inequality}
		\iSum \sst \phit \nnFxt \leq \Dz + \frac \sm 2 \iSum \sst^2.
	\end{align}
	Next we apply \Cref{lem:app.remove_normalisation} to get
	\begin{align*}
		\Exp{\phit \nnFxt} 
		\geq \Exp{\nnFxt - 2 \norm{\gt - \nFxt}}
		\geq \Exp{\nnFxt} - 2 \sigt,
	\end{align*}
	where we applied our assumption $\sigt \geq \Exp{\norm{\gt - \nFxt}}$ in the last inequality.
	Therefore, by taking expectations in \eqref{eq:app.unified_expectation.descent_inequality}, we get
	\begin{align*}
		\iSum \sst \Exp{\nnFxt} \leq \Dz + \frac \sm 2 \iSum \sst^2 + 2 \iSum \sst \sigt.
	\end{align*}
	Dividing by $\tauSum \sstau$ yields the claim.
\end{proof}

\subsubsection{Proof of Proposition 1}
\label{sec:app.missing.exp.prop_1}

Now we are ready to prove the full version of \Cref{prop:main.param_agnostic_general}.

\begin{proof}[Proof of \Cref{prop:main.param_agnostic_general}.]
	To shorten the notation we write $\ssl \coloneqq \ssi \li^{-r}$ and $\Bl \coloneqq \ceil*{\max\set{1, B\li^q}}$. Remember, that we are considering the mini-batch gradient-estimator
	\begin{align*}
		\gt = \frac 1 {\Bl} \sum_{j = 1}^{\Bl} \nabla \ora \pare{\xt, \xit^{(j)}}.
	\end{align*}
	We start by controlling the (conditional) expected deviation of $\gt$ from $\nFxt$ using \Cref{lem:app.von_Bahr_and_Essen}. Let $x \in \R^d$ and define $X_j(x) \coloneqq \nabla f \pare{x, \xitj} - \nF \pare x$ for all $j \in \left[\Bl\right]$. Now note that $X_1(x), \dots, X_{\Bl}(x)$ are independent random variables with mean zero and hence a Martingale Difference Sequence (MDS). Furthermore note that $\Exp{\norm{X_j(x)}^p} \leq \sigma^p$ by \pBCM\ and we can hence apply \Cref{lem:app.von_Bahr_and_Essen} to get
	\begin{align}\label{eq:app.missing.exp.param_agnostic_general.pointwise}
		g(x) 
		\coloneqq\Exp{\norm{\sum_{j=1}^{\Bl} X_j(x)}^p}
		\leq 2 \sum_{j=1}^{\Bl} \Exp{\norm{X_j(x)}^p}
		\leq 2 \Bl \varsym^p.
	\end{align}
	
	Next we calculate
	\begin{align}\label{eq:app.missing.exp.param_agnostic_general.cond_exp_dev.temp}
	\begin{split}
		\Exp[\xt]{\norm{\gt - \nFxt}} 
		= \Exp[\xt]{\norm{\frac 1 {\Bl} \sum_{j=1}^{\Bl} \pare{\nf{\xt, \xit^{(j)}} - \nFxt}}}\\
		\leq \frac 1 {\Bl} \Exp[\xt]{\norm{ \sum_{j=1}^{\Bl} \pare{\nf{\xt, \xit^{(j)}} - \nFxt}}^p}^{\nicefrac 1 p}
	\end{split}
	\end{align}
	where we applied Jensen in the last inequality. Next define
	\begin{align*}
		Y = \pare{\xitGen{1}, \dots, \xitGen{\Bl}} \quad \text{and} \quad
		h\pare{\xt,Y} &= \norm{ \sum_{j=1}^{\Bl} \pare{\nf{\xt, \xitj} - \nFxt}}^p.
	\end{align*}
	and note that $\xt$ and $Y$ are independent. Hence we may apply \Cref{lem:app.technical.cond_exp} which yields
	\begin{align}\label{eq:app.missing.exp.param_agnostic_general.cond_exp_dev}
		\Exp[\xt]{\norm{\gt - \nFxt}} 
		\stackrel{\eqref{eq:app.missing.exp.param_agnostic_general.cond_exp_dev.temp}}{\leq}
		\frac 1 {\Bl} \Exp[\xt]{h\pare{\xt,Y}}^{\nicefrac 1 p}
		\stackAlign{\text{Lem. }\ref{lem:app.technical.cond_exp}}{=}
		\frac 1 {\Bl} g \pare{\xt}^{\nicefrac 1 p}
		\stackrel{\eqref{eq:app.missing.exp.param_agnostic_general.pointwise}}{\leq}
		2 \frac{\sigma}{\Bl^{\frac{p-1} p }}
	\end{align}
	almost surely. By the tower property we get $\Expnorm{\gt - \nFxt} = \Exp{\Exp[\xt]{\norm{\gt - \nFxt}}} \leq 2 \sigma \Bl^{- \frac{p-1} p }$ and plugging into \eqref{eq:app.missing.exp.unified_exp.constant} yields
	\begin{align*}
		\frac 1 \li \iSum \Expnorm{\nFxt}
		\leq \frac{\Dz}{\ssl \li} + \frac{\ssl \sm}{2} + \frac{4 \varsym}{\Bl^{\frac{p-1}{p}}}.
	\end{align*}
	Using the definitions of $\ssl$ and $\Bl$ we get
	\begin{align}\label{eq:app.missing.exp.param_free.rate}
		\frac 1 \li \iSum \Expnorm{\nFxt}
		\leq \frac{\Dz}{\ssi \li^{1-r}} + \frac{\ssi \sm}{2\li^r} + \frac{4 \varsym}{\ceil*{\max\set{1, B\li^q}}^{\frac{p-1}{p}}}
		\leq \frac{\Dz}{\ssi \li^{1-r}} + \frac{\ssi \sm}{2\li^r} + \frac{4 \varsym}{\max\set{1, B\li^q}^{\frac{p-1}{p}}}
	\end{align}
	This implies an iteration complexity of
	\begin{align*}
		\Oc\pare{
			\pare{\frac{\Dz}{\ssi \eps}}^{\frac 1 {1-r}} 
			+ \pare{\frac{\ssi \sm}{\eps}}^{\frac 1 r} 
			+ \frac 1 {B^{\nicefrac 1 q}}\pare{\frac{\varsym}{\eps}}^{\frac{p}{q\pare{p-1}}}
		}
	\end{align*}
	and hence a sample complexity of
	\begin{align}\label{eq:app.missing.exp.param_free.sample_complexity}
        \Oc\pare{
        \pare{\frac{\Dz}{\ssi \eps}}^{\frac 1 {1-r}} 
        + \pare{\frac{\ssi \sm}{\eps}}^{\frac 1 r}
        + B\pare{\frac{\Dz}{\ssi \eps}}^{\frac {1+q} {1-r}} 
        + B\pare{\frac{\ssi \sm}{\eps}}^{\frac {1+q} r}
        + \frac 1 {B^{\frac 1 q}} \pare{\frac{\varsym}{\eps}}^{\frac{p\pare{1+q}}{q\pare{p-1}}}
        }.
    \end{align}
	This completes the proof.
\end{proof}

\subsubsection{Proof of Corollary 3}
\label{sec:app.missing.exp.cor_3}

Finally we provide a slightly more refined analysis to prove \Cref{cor:main.optimal_complexity}.

\begin{proof}[Proof of \Cref{cor:main.optimal_complexity}]
	Applying \Cref{prop:main.param_agnostic_general} to our choice of parameters yields
	\begin{align*}
		\frac 1 \li \iSum \Expnorm{\nFxt} 
		\leq 2 \frac{\sqrt{\Dz \sm}}{\sqrt \li} + \frac{4\sigma}{\max \set{1,\pare{\frac{\varsym^2 \li}{\Dz \sm}}^{\frac{p}{2p-2}}}^{\frac{p-1} p}}.
	\end{align*}
	We proceed with a case distinction.\\
	\textbf{Case 1: $\frac{\varsym^2 \li}{\Dz \sm} \leq 1$.} In this case we get $\varsym \leq \frac{\sqrt{\Dz \sm}}{\sqrt \li}$ and hence
	\begin{align*}
		{4\sigma}\pare{\max \set{1,\pare{\frac{\varsym^2 \li}{\Dz \sm}}^{\frac{p}{2p-2}}}}^{-\frac{p-1} p}
		= 4 \varsym
		\leq 4 \frac{\sqrt{\Dz \sm}}{\sqrt \li}.
	\end{align*}
	\textbf{Case 2: $\frac{\varsym^2 \li}{\Dz \sm} > 1$.} We calculate
	\begin{align*}
		4\sigma \pare{\max \set{1,\pare{\frac{\varsym^2 \li}{\Dz \sm}}^{\frac{p}{2p-2}}}}^{-\frac{p-1} p}
		= 4 \varsym \pare{\frac{\varsym^2 \li}{\Dz \sm}}^{-\frac 1 2}
		= 4 \frac{\sqrt{\Dz \sm}}{\sqrt \li}.
	\end{align*}
	This implies an iteration complexity of $\Oc\pare{\Dz\sm\eps^{-2}}$ and hence a sample complexity of
	\begin{align*}
		\Oc\pare{\Dz \sm \eps^{-2} \cdot \ceil*{\max \set{1,\pare{\frac{ \varsym^2}{\eps^2}}^{\frac{p}{2p-2}}}}}
		= \Oc\pare{\frac{\Dz \sm}{\eps^2} + \frac{\Dz \sm}{\eps^2}\pare{\frac{\varsym}{\eps}}^{\frac p {p-1}} }.
	\end{align*}
	
\end{proof}


\subsection{High-Probability Upper-Bounds}
\label{sec:app.missing.hp}

This subsection contains the proofs for our high-probability results. We start off with the proof of \Cref{thm:main.unified_hp}. The proof hinges on the observation that $\frac{\nFxt^\top \gt}{\nnFxt \norm{\gt}} \in [-1,1]$ is bounded and hence concentrates well. This will allow us to apply \Cref{lem:app.technical.conc_ineq_MDS} to get the mild $\logood$ dependence. 

\subsubsection{Proof of Theorem 4}
\label{sec:app.missing.hp.thm_4}

\begin{proof}[Proof of \Cref{thm:main.unified_hp}.]
	Let $\phit \coloneqq \frac{\nFxt^\top \gt}{\nnFxt \norm{\gt}}$ denote the cosine between $\nFxt$ and $\gt$. Then, by \Cref{lem:app.descent_lemma}, we have
	\begin{align*}
		\iSum \sst \phit \nnFxt \leq \Dz + \frac \sm 2 \iSum \sst^2.
	\end{align*}
	Next, we use the fact that $\phit$ is bounded and hence sharply concentrates around its (conditional) expectation. Formally, let $\psit \coloneqq \Exp[\Fctm]{\phit}$ and note that $\Dt \coloneqq -\sst \pare{\phit - \psit}\nnFxt$ is a martingale difference sequence with respect to $\pare{\Fct}_{\ii \in \N}$. Furthermore, noting that
	\begin{align*}
		\exp\pare{\frac{\Dt^2}{4\sst^2\nnFxt^2}} = \exp\pare{\frac{\pare{\phit-\psit}^2}{4}} \leq e
	\end{align*}
	implies that we may apply \Cref{lem:app.technical.conc_ineq_MDS} with $G_t^2 = 4\sst^2\nnFxt^2$. Doing so yields, for all $\lambda > 0$,
	\begin{align*}
		\iSum\sst \pare{\psit - 3 \lambda \sst \nnFxt} \nnFxt 
		\leq \Dz + \frac \sm 2 \iSum \sst^2 + \frac 1 \lambda \logood
	\end{align*}
	with probability at least $1-\delta$. Using \lsmooth\ we get $\nnFxt \leq \norm{\nF \pare{\xfin}} + \sm \sum_{\tau = \fin}^{\ii - 1} \sstau$ and hence choosing $\lambda \coloneqq \frac 1 {6 \pare{\ssmax \norm{\nF\pare{\xfin}} + C_\li \sm}}$ yields, with probability at least $1-\delta$,
	\begin{align}\label{eq:unified_HP_theorem_critical_ineq}
		\iSum \sst \pare{\psit - \frac 1 2}\nnFxt
		\leq \Dz + \frac \sm 2 \iSum \sst^2 + 6  \pare{\ssmax\norm{\nF\pare{\xfin}} + C_\li \sm} \logood.
	\end{align}
	Finally we are left with the challenge of guaranteeing that $\psit$ is \emph{large enough}. Therefore we use \Cref{lem:app.remove_normalisation} to get $\psit \nnFxt = \Exp[\Fctm]{\frac{\nFxt^\top\gt}{\norm{\gt}}} \geq \nnFxt - 2 \Exp[\Fctm]{\norm{\gt - \nFxt}} = \nnFxt - 2 \sigt$ and hence
	\begin{align*}
		\frac 1 2 \iSum \sst \nnFxt
		\leq \Dz + \frac \sm 2 \iSum \sst^2 + 2 \iSum \sst \sigt + 6 \pare{\ssmax\norm{\nF\pare{\xfin}} + C_\li \sm} \logood.
	\end{align*}
	Dividing by $\frac 1 2 \tauSum \sstau$ yields the claim.
\end{proof}

\subsubsection{Proof of Corollary 5}
\label{sec:app.missing.hp.cor_5}

Next we apply \Cref{thm:main.unified_hp} to derive the high-probability result for tuned \batchNsgdName.

\begin{proof}[Proof of \Cref{cor:main.optimal_hp}]
To shorten the notation we write $\sst \equiv \ssi$ and $\Bt \equiv B$. First note that $\xt$ is $\sigma\pare{\xt}\subseteq \Fctm$ measurable and $\xitGen 1, \dots , \xitGen B$ are independent of $\Fctm$. In particular we have $\Exp[\Fctm]{\norm{\gt - \nFxt}} = \Exp[\xt]{\norm{\gt - \nFxt}}$ and hence may apply \eqref{eq:app.missing.exp.param_agnostic_general.cond_exp_dev} to get
\begin{align}\label{eq:app.missing.hp.batch_cond_mom_deviation_bound}
		\Exp[\Fctm]{\norm{\gt - \nFxt}} 
		= \Exp[\xt]{\norm{\gt - \nFxt}}
		\stackrel{\eqref{eq:app.missing.exp.param_agnostic_general.cond_exp_dev}}{\leq}
		\frac{2\varsym}{B^{\frac{p-1}{p}}},
\end{align}
Plugging into \Cref{thm:main.unified_hp} now yields
\begin{align}\label{eq:app.missing.optimal_hp.prelim}
\begin{split}
	\frac 1 \li \iSum \nnFxt
	\leq&\ \frac{2\Dz}{\ssi \li} + \Aa + 8\varsym B^{-\frac{p-1}p} + 12\pare{\frac{\norm{\nF \pare{\xfin}}}{\li} + \ssi \sm}\logood\\
	=&\ 2 \frac{\sqrt{\Dz \sm}}{\sqrt \li} + \frac{\sqrt{\Dz \sm}}{\sqrt{\li}}
	+ 8\varsym B^{-\frac{p-1}p} + 12\pare{\frac{\norm{\nF \pare{\xfin}}}{\li} + \frac{\sqrt{\Dz \sm}}{\sqrt{\li}}}\logood \\
	\leq&\ \pare{3 + 30 \logood} \frac{\sqrt{\Dz \sm}}{\sqrt\li} + 8\varsym B^{-\frac{p-1}p},
\end{split}
\end{align}
where we used $\norm{\nF \pare{\xfin}} \leq \sqrt{2\Dz\sm}$ in the last inequality. We now proceed with a case distinction.\\
\textbf{Case 1: $B = 1$.} This implies $\varsym \leq \sqrt{\frac{\Dz \sm} \li}$ and hence
\begin{align*}
	\frac 1 \li \iSum \nnFxt
	\leq \pare{11 + 30 \logood} \frac{\sqrt{\Dz \sm}}{\sqrt \li}.
\end{align*}
\textbf{Case 2: $B = \pare{\frac{\varsym^2\li}{\Dz \sm}}^{\frac{p}{2p-2}}$.} In this case we have $\varsym B^{\frac{1-p}p} = \sqrt{\frac{\Dz \sm}{\li}}$ and plugging into \eqref{eq:app.missing.optimal_hp.prelim} yields 
\begin{align*}
	\frac 1 \li \iSum \nnFxt
	\leq \pare{11 + 30 \logood} \frac{\sqrt{\Dz \sm}}{\sqrt \li}.
\end{align*}
This finishes the convergence result. To prove the oracle complexity, note that each iteration requires $1$ and $\pare{\frac{\varsym^2\li}{\Dz \sm}}^{\frac{p}{2p-2}}$ oracle calls in Case 1 and 2 respectively. To reach an $\eps$-stationary point, $\Oc\pare{\Dz \sm \eps^{-2}\logood^2}$ iterations are required. Plugging into the oracle complexity per iteration yields the second claim.
\end{proof}

\subsubsection{Parameter-Free High-Probability}
\label{sec:app.missing.hp.param_free}

Similar to \Cref{prop:main.param_agnostic_general}, we can also derive a parameter-free high-probability result for \batchNsgd.

\begin{corollary}
    Assume \lb, \lsmooth\ and \pBCM\ with $p \in (1, 2]$. Furthermore let $\ssi, B, q > 0$ and $\delta, r \in (0,1)$. Then the iterates generated by \batchNsgd\ with parameters $\sst \equiv \ssi \li^{- r}$ and $\Bt \equiv \ceil{\max\set{1, B\li^q}}$ satisfy, with probability at least $1 - \delta$,
    \begin{align*}
        \frac 1 \li \iSum \nnFxt \leq 
        \frac{2\Dz}{\ssi \li^{1-r}} + \frac{\ssi \sm}{\li^r} \pare{1 + 12\logood} + 17\frac{\sqrt{\Dz \sm}}{\li}\logood + \frac{8\varsym}{\max\set{1, B\li^q}^{\frac{p-1}p}}
	\end{align*}
	In particular, the sample complexity is bounded by $\Oct \pare{\pare{\frac \Dz \eps}^{\frac{1+q}{1-r}} + \pare{\frac \sm \eps}^{\frac{1+q}{r}} + \pare{\frac{\varsym}{\eps}}^{\frac{p\pare{1+q}}{q\pare{p-1}}}}$.
\end{corollary}

\begin{proof}
    To shorten the notation we write $\sst \equiv \ssl$ and $\Bt \equiv \Bl$. First we apply \eqref{eq:app.missing.hp.batch_cond_mom_deviation_bound} to get $\Exp[\Fctm]{\norm{\gt - \nFxt}} \leq \frac{2\varsym}{B^{\frac{p-1}{p}}}$ and plugging into \Cref{thm:main.unified_hp} yields
\begin{align}\label{eq:app.missing.hp.param_free.prelim}
\begin{split}
	\frac 1 \li \iSum \nnFxt
	\leq&\ \frac{2\Dz}{\ssl \li} + \ssl \sm + 8\varsym \Bl^{-\frac{p-1}p} + 12\pare{\frac{\norm{\nF \pare{\xfin}}}{\li} + \ssl \sm}\logood\\
	\leq&\ \frac{2\Dz}{\ssi \li^{1-r}} + \frac{\ssi \sm}{\li^r} \pare{1 + 12\logood}
	+ \frac{8\varsym}{\ceil{\max\set{1, B\li^q}}^{\frac{p-1}p}} + 17\frac{\sqrt{\Dz \sm}}{\li}\logood \\
	\leq&\ \frac{2\Dz}{\ssi \li^{1-r}} + \frac{\ssi \sm}{\li^r} \pare{1 + 12\logood} + 17\frac{\sqrt{\Dz \sm}}{\li}\logood
	+ \frac{8\varsym}{\max\set{1, B\li^q}^{\frac{p-1}p}}
\end{split}
\end{align}
where we used $\norm{\nF \pare{\xfin}} \leq \sqrt{2\Dz\sm}$ in the second inequality. This corresponds to an iteration complexity of
\begin{align*}
    \Oct\pare{
    \pare{\frac{\Dz}{\ssi \eps}}^{\frac 1 {1-r}} 
    + \pare{\frac{\ssi \sm}{\eps}}^{\frac 1 r}
    + \frac 1 {B^\frac 1 q} \pare{\frac{\varsym}{\eps}}^{\frac{p}{q\pare{p-1}}}
    },
\end{align*}
where we used $\frac{\sqrt{\Dz \sm}}{\eps} \leq \pare{\frac{\sqrt \Dz}{\ssi \eps}}^{\frac 1 {1-r}} + \pare{\frac{\ssi\sqrt{\sm }}{\eps}}^{\frac 1 r} = \Oc\pare{\pare{\frac{\Dz}{\ssi \eps}}^{\frac 1 {1-r}} + \pare{\frac{\ssi\sm}{\eps}}^{\frac 1 r}}$ by Young's inequality. Finally, this implies a sample complexity of
\begin{align*}
    \Oct\pare{
    \pare{\frac{\Dz}{\ssi \eps}}^{\frac 1 {1-r}} 
    + \pare{\frac{\ssi \sm}{\eps}}^{\frac 1 r}
    + B\pare{\frac{\Dz}{\ssi \eps}}^{\frac {1+q} {1-r}} 
    + B\pare{\frac{\ssi \sm}{\eps}}^{\frac {1+q} r}
    + \frac 1 {B^{\frac 1 q}} \pare{\frac{\varsym}{\eps}}^{\frac{p\pare{1+q}}{q\pare{p-1}}}
    }.
\end{align*}
\end{proof}

}
\clearpage

\section{UPPER-BOUNDS FOR NSGD WITH MOMENTUM}
\label{sec:app.nsgdm}
In this section we discuss the version of our results for NSGD with momentum (\nsgdmName), i.e., \nsgd\ with the gradient estimator
\begin{align}\label{eq:app.nsgdm}
    \gt = \momt \gtm + \pare{1-\momt}\nf{\xt,\xit},
\end{align}
where $\iterationg 0 = 0$. 
Throughout this section we use the notation $\alpha_t := 1 - \beta_t$ an  $\momprod a b \coloneqq \prod_{\kappa = a}^b \iterationmom \kappa$. 

We first derive a deviation bound for $\gt$ from $\nFxt$, similar to \eqref{eq:app.missing.exp.param_agnostic_general.cond_exp_dev} but for the momentum estimator, generalising the bound in \citep{MomentumImprovesNSGD_Cutkosky_2020} to $p < 2$.

\begin{lemma}\label{lem:app.nsgdm.moment_deviation_bound}
	Let $\iterationmom \fin = 0$ and assume \lsmooth, \pBCM\ with $p \in (1, 2]$. Then the iterates generated by \nsgdm\ satisfy
	\begin{align*}
		\Expnorm{\gt -\nFxt} 
		\leq
		\sm \sum_{\tau = 2}^\ii \iterationss {\tau -1} \momprod \tau \ii
		+ 2\varsym \pare{ \sum_{\tau = 1}^\ii \pare{\momprod {\pare{\tau + 1}} \ii \pare{1-\iterationmom \tau}}^p}^{\nicefrac 1 p}.
	\end{align*}
\end{lemma}

\begin{proof}
	{ 
	\renewcommand{\gat}{\mu_\ii} \renewcommand{\gatm}{\mu_{\ii - 1}}
	To simplify notation we first define
	\begin{align*}
		\gat &\coloneqq \gt - \nFxt,\\
		\epst &\coloneqq \nf{\xt, \xit} - \nFxt,\\
		\St &\coloneqq \nF \pare \xtm - \nF \pare \xt.
	\end{align*}
	Now we calculate
	\begin{align*}
		\gt 
		&= \momt \gtm + \pare{1-\momt}\nf{\xt, \xit} \\
		&= \momt \pare{\nF \pare \xkm + \gatm} + \pare{1- \momt} \pare{\epst + \nFxt}\\
		&= \nFxt + \pare{1 - \momt} \epst + \momt \St + \momt \gatm
	\end{align*}
	and unrolling yields
	\begin{align*}
		\gat
		&= \momprod 2 \ii \iterationga \fin 
		+ \sum_{\tau = 2}^\ii \momprod {\pare{\tau + 1}} \ii \iterationa \tau \iterationeps \tau
		+ \sum_{\tau = 2}^\ii \momprod \tau \ii \iterationS \tau
		=  \sum_{\tau = 1}^\ii \momprod {\pare{\tau + 1}} \ii \iterationa \tau \iterationeps \tau
		+ \sum_{\tau = 2}^\ii \momprod \tau \ii \iterationS \tau,
	\end{align*}
	where we used $\iterationmom \fin = 0$ in the second equality. Therefore
	\begin{align}\label{eq:app.nsgdm.momentum_deviation_bound.gat_bound}
		\Expnorm \gat
		&\leq \Expnorm{\sum_{\tau = 1}^\ii \momprod {\pare{\tau + 1}} \ii \iterationa \tau \iterationeps \tau}
		+ \sum_{\tau = 2}^\ii \momprod \tau \ii \Expnorm{\iterationS \tau}
		\leq
		\Exp{\norm{\sum_{\tau = 1}^\ii \momprod {\pare{\tau + 1}} \ii \iterationa \tau \iterationeps \tau}^p}^{\nicefrac 1 p}
		+ \sum_{\tau = 2}^\ii \momprod \tau \ii \Expnorm{\iterationS \tau},
	\end{align}
	where we applied Jensen in the second inequality. The second sum can be upper bounded by $\sm \sum_{\tau = 2}^\ii \iterationss {\tau - 1} \momprod \tau \ii$. To control the first sum we want to apply \Cref{lem:app.von_Bahr_and_Essen}. 
	
	\newcommand{\Ctau}{C_\tau}
	\newcommand{\Xtau}{X_\tau}
	Therefore, to simplify notation, let $\Ctau \coloneqq \momprod {\pare{\tau + 1}} \ii \iterationa \tau$ and $\Xtau \coloneqq  \Ctau \iterationeps \tau$. To check whether $X_1, \dots, X_\ii$ satisfies the assumptions of \Cref{lem:app.von_Bahr_and_Essen}, first note that, for all $\tau \in [\ii]$,
	\begin{align}\label{eq:app.nsgdm.application_vBE.start}
		\Exp[X_1, \dots, X_{\tau - 1}]{\Xtau}
		= \Ctau \Exp[X_1, \dots, X_{\tau - 1}]{\nf{\iterationx \tau, \xi_{\tau}} - \nF\pare{\iterationx \tau}}
	\end{align}
	and furthermore, as $\iterationx \tau$ is $\sigma(X_1, \dots, X_{\tau - 1})$ measurable and $\xi_\tau$ independent of $X_1, \dots, X_{\tau - 1}$, we have\footnote{Rigorously, this follows from \Cref{lem:app.technical.remove_unnecessary_cond}.}
	\begin{align*}
		\Exp[X_1, \dots, X_{\tau - 1}]{\nf{\iterationx \tau, \xi_{\tau}} - \nF\pare{\iterationx \tau}}
		= \Exp[\iterationx \tau]{\nf{\iterationx \tau, \xi_{\tau}} - \nF\pare{\iterationx \tau}}
		= 0,
	\end{align*}
	where we applied our unbiasedness assumption in conjunction with \Cref{lem:app.technical.cond_exp} in the last equality. By a similar argument, using \pBCM, we get
	\begin{align*}
		\Exp{\norm{\Xtau}^p} 
		= \Ctau^p \Exp{\Exp[\iterationx \tau]{\norm{\nf{\iterationx \tau, \xi_{\tau}} - \nF\pare{\iterationx \tau}}^p}}
		\leq \Ctau^p \varsym^p < \infty.
	\end{align*}
	Hence we may apply \Cref{lem:app.von_Bahr_and_Essen} to get
	\begin{align}\label{eq:app.nsgdm.application_vBE.end}
		\Exp{\norm{\sum_{\tau = 1}^\ii \momprod {\pare{\tau + 1}} \ii \iterationa \tau \iterationeps \tau}^p}^{\nicefrac 1 p}
		\leq
		\pare{2\sum_{\tau = 1}^\ii \Ctau^p \varsym^p}^{\nicefrac 1 p}
		\leq 
		2 \varsym \pare{\sum_{\tau = 1}^\ii \pare{\momprod {\pare{\tau + 1}} \ii \iterationa \tau}^p}^{\nicefrac 1 p}
	\end{align}
	Combining these bounds with \eqref{eq:app.nsgdm.momentum_deviation_bound.gat_bound} yields
	\begin{align*}
		\Expnorm \gat 
		\stackrel{\eqref{eq:app.nsgdm.momentum_deviation_bound.gat_bound}}{\leq} 
		\Exp{\norm{\sum_{\tau = 1}^\ii \momprod {\pare{\tau + 1}} \ii \iterationa \tau \iterationeps \tau}^p}^{\nicefrac 1 p}
		+ \sum_{\tau = 2}^\ii \momprod \tau \ii \Expnorm{\iterationS \tau}
		\leq
		2 \varsym \pare{\sum_{\tau = 1}^\ii \pare{\momprod {\pare{\tau + 1}} \ii \iterationa \tau}^p}^{\nicefrac 1 p}
		+ \sm \sum_{\tau = 2}^\ii \iterationss {\tau - 1} \momprod \tau \ii
	\end{align*}
	and hence the claim.
	}
\end{proof}

\subsection{Parameter-Free}
\label{sec:app.nsgdm.param_free}

Next we derive the \nsgdm\ counterpart to the parameter-free result \eqref{eq:main.optimal_param_free_result}. Additionally, the result is phrased for decreasing stepsizes, outlining how results can be extended to those.

\begin{corollary}[Parameter-Agnostic Convergence]\label{lem:app.nsgdm.parameter_free}
	Let $\li \geq 3$ and assume \lb, \lsmooth\ and \pBCM\ with $p \in (1, 2]$. Then the iterates generated by \nsgd\ with $\gt = \momt \gtm + \pare{1-\momt} \nf{\xt, \xit}$ and parameters $\momt = 1 - \ii^{-\nicefrac 1 2}$ and $\sst = \ssi \ii^{- \nicefrac 3 4}$ satisfy
	\begin{align*}
		\frac 1 \li \iSum \Exp{\nnFxt} 
		\leq
		\frac{\frac{\Dz}{\ssi} + 120 \Aa \log \pare \li
			+ 120 \varsym \frac{4p}{2-p}\pare{\li^{\frac{2-p}{4p}} - 1}} {\li^{\frac 1 4}}.
	\end{align*}
	In particular, this corresponds to a rate of convergence of $\Oct \pare{\pare{\Dz + \sm}{\li^{-\nicefrac 1 4}} + {\varsym}{\li^{-\frac{p-1}{2p}}}}$ and hence a sample complexity of $\Oct \pare{\frac{\Dz^4 + \sm^4}{\eps^4} + \pare{\frac{\varsym}{\eps}}^{\frac{2p}{p-1}}}$.
\end{corollary}

The proof follows similar steps as in \citep{MomentumImprovesNSGD_Cutkosky_2020}, but requires additional attention to the noise term to handle the case $p<2$.

\begin{proof}
	\newcommand{\pExpo}{\frac{5p-2}{4p}}
	To shorten notation we define $r \coloneqq \nicefrac 3 4, q \coloneqq \nicefrac 1 2$, and hence $\sst = \ssi \ii^{-r}, \momt = 1 - \ii^{-q}$. Furthermore let $\sigt \coloneqq \Expnorm{\gt - \nFxt}$. From \Cref{prop:app.missing.exp.unified_expectation} we get
	\begin{align}\label{eq:app.nsgdm.parameter_free.applying_unified}
		\begin{split}
			\iSum \frac{\sst}{\tauSum \sstau} \Exp{\nnFxt}
			&\leq \pare{\iSum \sst}^{-1}\pare{\Dz + \frac \sm 2 \iSum \sst^2 + 2 \iSum \sst \sigt}\\
			&\leq \li^{r-1}\pare{\frac {\Dz} {\ssi} + \frac 3 2 \Aa + 2\iSum \ii^{-r}\sigt},
		\end{split}
	\end{align}
	where we used $\iSum \sst \geq \ssi \li^{1-r}$ and $\iSum \sst^2 \leq 3 \ssi^2$ in the second inequality. To control the third term, we apply \Cref{lem:app.nsgdm.moment_deviation_bound} and \Cref{lem:app.technical.dec_eq_const} to get
	\begin{align*}
		\iSum \ii^{-r} \sigt 
		&\leq 4\exp\pare{\frac{1}{1-q}} \iSum \pare{\varsym \ii^{-r-q\frac{p-1}p} + \Aa \ii^{-2r+q}}\\
		&= 4e^2 \iSum \pare{\varsym \ii^{-\pExpo} + \Aa\ii^{-1}}.
		\\ &\leq 4e^2 \pare{\varsym\iSum \ii^{-\pExpo} + \Aa\pare{1+\log \pare \li}}.
	\end{align*}
	In order to bound $\iSum \ii^{-\pExpo}$ we note that $\pExpo = 1$ iff $p=2$ and hence
	\begin{align*}
		\iSum \ii^{-\pExpo}
		\leq 1 + \int_1^\li \ii^{-\pExpo} dt
		\leq
		\begin{cases}
			1 + \log \pare \li, & \text{if } p = 2\\
			1 + \frac{1}{1 - \pExpo} \pare{\li^{1-\pExpo} - 1}, & \text{otherwise.}
		\end{cases}
	\end{align*}
	Now note that, due to L'Hôspital, $\lim_{q\to 1}\frac{1}{1-q}\pare{\li^{1-q}-1} = \log\pare\li$ and hence we can unify the cases by writing the second expression and using continuous extensions.
	Plugging into \eqref{eq:app.nsgdm.parameter_free.applying_unified} yields
	\begin{align*}
		 \iSum \frac{\sst}{\tauSum \sstau} \Exp{\nnFxt}
		\leq&\ \li^{r-1} \pare{\frac{\Dz}{\ssi}
			+ 8e^2 \Aa \pare{1+\log \pare \li} 
			+ 8e^2 \varsym \pare{1+\frac{4p}{2-p}\pare{\li^{\frac{2-p}{4p}}-1} }}\\
		\leq&\ 
		\li^{-\nicefrac 1 4} \pare{
			\frac{\Dz}{\ssi} 
			+ 120 \Aa \log \pare \li
			+ 120 \varsym \frac{4p}{2-p}\pare{\li^{\frac{2-p}{4p}} - 1}
		},
	\end{align*}
	where we used that $\frac{4p}{2-p}\pare{\li^{\frac{2-p}{4p}}-1} \geq 1$ for $\li \geq 3$ in the last inequality. The other statements follow from the observation  $\lim_{q\to 1}\frac{1}{1-q}\pare{\li^{1-q}-1} = \log\pare\li$.
\end{proof}

\subsection{Optimal Sample Complexity}
\label{sec:app.nsgdm.tuned}

Finally we provide the \nsgdm\ version of \Cref{cor:main.optimal_complexity}.

\begin{corollary}[Optimal Oracle Complexity]\label{lem:app.nsgdm.tuned}
	Assume \lb, \lsmooth\ and \pBCM\ with $p \in (1, 2]$. Then the iterates generated by \nsgdm\ with parameters $\iterationmom \fin \coloneqq 0, \momt \equiv \beta \coloneqq 1 - \min\set{1, \pare{\frac{\Dz \sm}{\varsym^2\li}}^{\frac{p}{3p-2}}}$ for $\ii \geq 2$ and $\sst \equiv \sqrt{\frac{\Dz \pare{1-\momentum}}{\sm \li}}$ satisfy
	\begin{align*}
		\frac 1 \li \iSum \Exp{\nnFxt} \leq 6 \frac{\sqrt{\Dz \sm}}{\sqrt\li} + 6 \pare{\frac{\Dz \sm \varsym^{\frac p {p-1}} }{\li}}^{\frac{p-1}{3p-2}}.
	\end{align*}
	In particular, this corresponds to an oracle complexity of $\Oc \pare{\frac{\Dz \sm}{\eps^2} + \frac{\Dz \sm}{\eps^2} \pare{\frac \varsym \eps}^{\frac p {p-1}}}$.
\end{corollary}

\begin{proof}
	To shorten the notation we write $\sst \equiv \ssi, \momt \equiv \momentum$ and $\al \coloneqq 1 - \momentum$. Combining \eqref{eq:app.missing.exp.unified_exp.constant} with \Cref{lem:app.nsgdm.moment_deviation_bound} yields
	\begin{align}\label{eq:unified_E_NSGDM_opt_const_1}
		\begin{split}
			\frac 1 \li \iSum \Exp{\nnFxt} 
			&\leq \frac{\Dz}{\ssi \li} + \frac{\ssi \sm}{2} 
			+ 2 \varsym \al^{\frac{p-1} p} 
			+ \frac{2 \sm\ssi} \al\\
			&= \sqrt{\frac{\Dz \sm}{\al \li}} + \frac{\sqrt{\Dz \sm \al}}{2 \sqrt{\li}}
			+ 2 \varsym \al^{\frac{p-1} p} + 2\sqrt{\frac{\Dz \sm}{\al \li}}\\
			&\leq 4 \sqrt{\frac{\Dz \sm}{\al \li}} + 2 \varsym \al^{\frac{p-1} p}.
		\end{split}
	\end{align}
	\textbf{Case 1: $\al = 1$.} This implies $\varsym \leq \sqrt{\frac{\Dz \sm} \li}$ and hence
	\begin{align*}
		\frac 1 \li \iSum \Exp{\nnFxt}
		\leq 6 \sqrt{\frac{\Dz \sm}{\li}}.
	\end{align*}
	\textbf{Case 2: $\al = \pare{\frac{\Dz \sm}{\varsym^2\li}}^{\frac{p}{3p-2}}$.} Plugging into \eqref{eq:unified_E_NSGDM_opt_const_1} yields
	\begin{align*}
		\frac 1 \li \iSum \Exp{\nnFxt}
		\leq 4 \varsym^{\frac p {3p-2}} \pare{\frac{\Dz \sm} {\li}}^{\frac{p-1}{3p-2}} 
		+ 2 \varsym^{\frac p {3p-2}} \pare{\frac{\Dz \sm} {\li}}^{\frac{p-1}{3p-2}}
		= 6 \varsym^{\frac p {3p-2}} \pare{\frac{\Dz \sm} {\li}}^{\frac{p-1}{3p-2}}. 
	\end{align*}
	Therefore we get
	\begin{align*}
		\frac 1 \li \iSum \Exp{\nnFxt}
		\leq 6 \max \set{\sqrt{\frac{\Dz \sm}{\li}}, \varsym^{\frac p {3p-2}} \pare{\frac{\Dz \sm} {\li}}^{\frac{p-1}{3p-2}}}
	\end{align*}
	and hence the claim.
\end{proof}

Additionally, this result recover those in \cite{MomentumImprovesNSGD_Cutkosky_2020}\footnote{Note that the authors did not use $\iterationmom \fin = 0$, resulting in an additional term. However this term is not leading and hence does not affect the oracle complexity.} with improved constants when $p=2$. 

\subsection{Technical Difficulties of Proving High Probability Convergence}
\label{sec:app.nsgdm.difficulties}

In the previous section we showed that \Cref{eq:main.optimal_param_free_result} and \Cref{cor:main.optimal_complexity} also hold for \nsgdName with momentum. For Theorem \ref{cor:main.optimal_hp} on the other hand, while it still holds for time-varying and constant parameters, we were not able to prove the result for \nsgdm. We shortly want to discuss the technical difficulty of extending \Cref{cor:main.optimal_hp} to the momentum version.

The proof of \Cref{thm:main.unified_hp} hinges on two parts: Firstly, one shows that the angle $\phit$ sharply concentrates around its conditional expectation $\psit = \Exp[\Fctm]{\phit}$. This step only requires the boundedness of $\phit$ and is hence applicable for both the minibatch and momentum version of \nsgdName. In the next step however, we have to lower bound $\psit$. Our current proof technique --- and to some extend intuition --- tells us that such lower bounds involves the term 
\begin{align}\label{eq:app.missing.hp.issue_term}
	\Exp[\Fctm]{\norm{\gt - \nFxt}}.
\end{align}
In the case of minibatch \nsgdName, $\gt$ only depends on randomness sampled in iteration $\ii$, and \eqref{eq:app.missing.hp.issue_term} can hence be upper bounded by a constant as seen in \eqref{eq:app.missing.exp.param_agnostic_general.cond_exp_dev}. However, in the case of \nsgdName with momentum, $\gt$ consists of random samples from all previous iterations. This results in \eqref{eq:app.missing.hp.issue_term} being a random variable instead, and it is not clear how to uniformly control it. Our empirical evidence indicates that an extension to \nsgdm\ might not be possible since quantiles of average gradient norms of \nsgdm\ behave super-linearly in $\log(1/\delta)$, see \Cref{sec:experiments,sec:appendix:verify_hp} for more details.
\clearpage

\section{LOWER-BOUNDS FOR NSGD}
\label{sec:app.optimality}
In this section we prove that \Cref{prop:main.param_agnostic_general} is tight and the sample complexity achieved in \Cref{eq:main.optimal_param_free_result} is optimal, in the sense that no other choice of parameters can lead to a uniformly better guarantee for \batchNsgd\ when problem-dependent parameters are unknown. To do so, we first derive a lower bound for the deterministic setting which might be of independent interest. Afterwards, we will equip this hard function with a stochastic gradient oracle to prove the lower bound for the stochastic setting.

\subsection{Deterministic Setting}
\label{sec:app.optimality.det}

The following Lemma derives a lower bound for \algname{NGD} with arbitrary stepsizes. We will use it afterwards to show optimality of our polynomial stepsize order.

{
\newcommand{\FcDzL}{\Fc_{\Dz, \sm}}
\newcommand{\ssli}{\iterationss \li}
\genCommandsName{tau}{Scalar}{\tau}
\genCommandsName{zeta}{Scalar}{\zeta}
\genCommands{A}{Scalar}
\begin{lemma}[Lower Bounds for Deterministic Setting] \label{lem:app.optimality.deterministic_arbitrary_ngd_optimiality}
	\newcommand{\algnot}{A_{\pare{\sst}}}
	Let $\FcDzL$ be the set of functions that satisfy \lb\ and \lsmooth\ and let $\eps > 0$. Denote \algname{NGD}\ with stepsizes $\sst \geq 0$ as $\algnot$. Then there exists a function $F \in \FcDzL$ such that the iterates of $\algnot$ satisfy
	\begin{align*}
		\norm{\nFxt} > \eps
	\end{align*}
	for all $\ii \in [\li^*]$, where $\li^* \coloneqq \inf \set{\li \in \N \suchthat \eps > \frac{\Dz - \frac \eps \sm}{2 \iSum \sst} + \frac{\sm \iSum \sst^2}{8 \iSum \sst}}$.
\end{lemma}
The proof extends ideas from \citep{NSGDM_LzLo2023Huebler} by also including the \lsmooth\ assumption into the function construction. 
\begin{proof}
	We first define
	\begin{align*}
		g_{\ssi} \colon \left[0, \ssi \right] \to \R, x \mapsto 
		\begin{cases}
			-2\eps + \sm x, & x \leq \frac {\ssi} 2\\
			-2\eps + {\ssi \sm} - \sm x, & x > \frac \ssi 2,
		\end{cases}
	\end{align*}
	which will correspond to the gradient of our constructed function between two consecutive points. To formalise this idea, define 
	\begin{align*}
		\At \coloneqq - 2\sst \eps + \frac{\sst^2 \sm}{4} = \int_0^{\eta} g_{\sst}(x) d x
	\end{align*}
	and $\li^*$ according to the statement. Furthermore, let  $\taut \coloneqq \sum_{\kappa = \fin}^{\ii - 1} \iterationss \kappa$. Then we define our hard function via its derivative
	\begin{align}\label{eq:app.optimality.deterministic_arbitrary_ngd_optimiality.derivative_construction}
		F'(x) \coloneqq
		\begin{cases}
			-2\eps, & x < 0\\
			g_{\sst} \pare{x-\taut}, & x \in [\taut, \tautp) ,\, \ii \in [\li^* - 1]\\
			-2\eps + L\pare{x-\iterationtau {\li^*}}, & x \in (\iterationtau {\li^*}, \iterationtau {\li^*} + \frac{2\eps}{L}]\\
			0, & \text{otherwise}
		\end{cases}
	\end{align}
	and $F(x) \coloneqq \Dz + \int_0^x F'(s) d s$. A sketch of $F$ and $F'$ can be seen in \Cref{fig:lower_bound_function}. Note that, by definition of $F$, we have 
	\begin{align*}
		F\pare{\taut}
		= \Dz + \sum_{\tau = 1}^{\ii-1} A_t 
		= \Dz - 2 \eps \sum_{\tau=1}^{\ii-1} \iterationss \tau + \frac{\sm}{4} \sum_{\tau = 1}^{\ii-1} \iterationss \tau^2
	\end{align*}
	and hence $F\pare{\taut} > \frac \eps \sm$ for all $\ii \leq \li^*$\footnote{Due to the shift $\ii - 1$ in the sum. Otherwise, if we would start the algorithm at $\iterationx 0$, $\ii \leq \li^* - 1$.} by our definition of $\li^*$. In particular we have $\inf_xF(x) \geq 0$ and hence $F(0) - \inf_x F(x) \leq \Dz$. Furthermore $F$ is $\sm$-smooth by definition and hence $F \in \FcDzL$. Next we will show that \algname{NGD}, when started at $\xfin = 0$, produces the iterates $\xt = \taut$. Therefore note that $\iterationtau \fin = 0 = \xfin$ and, assuming $\xt = \taut, \ii < \li^*$, we have $\xtp = \xt - \sgn(F'\atxt) \sst = \xt + \sst = \tautp$. By induction we hence get $\abs{F'\atxt} = \abs{F'\pare{\taut}}\equiv 2 \eps$ for all $\ii \in [\li^*]$, which completes the proof.
\end{proof}

\newcommand{\fDz}{f_{\Dz}}
The following Theorem is a consequence of \Cref{lem:app.optimality.deterministic_arbitrary_ngd_optimiality} and implies that $r = \nicefrac 1 2$ is the (only) optimal choice of polynomial stepsize decay in the deterministic setting. To formulate the result, we will use the complexity definition introduced by \citet{LowerBoundsDeterministic2020Carmon}, i.e., for an algorithm $A$, a function class $\Fc$ and $\eps > 0$ we define
\begin{align*}
    T_\eps\pare{A, \Fc} \coloneqq \sup_{F \in \Fc} \inf \set{ \ii \in \N \suchthat \nnFxt \leq \eps, \ \pare{\xt}_\tinN = A(F)}. 
\end{align*}

\begin{theorem}[Optimality in Deterministic Setting] \label{thm:app.optimality.deterministic_polynomial_ngd_optimiality}
	Let $\FcDzL$ be the set of functions that satisfy \lb\ and \lsmooth. Suppose $\eps \leq \frac{\Dz \sm}{2}$ and $\eps \leq \frac{\sqrt{\Dz \sm}}{3}$. Furthermore let $\ssi, r > 0$ and denote \algname{NGD}\ with decaying stepsizes $\sst = \ssi \, \ii^{-r}$ as $A_d^r$, and with constant stepsizes $\sst \equiv \ssi \, \li^{-r}$ as $A_c^r$. Then $A^r \in \set{A_d^r, A_c^r}$ satisfies
	\begin{align*}
		T_\eps\pare{A^r, \FcDzL} \geq 
		\pare{\frac{\ssi \sm}{4 \eps}}^{\frac 1 r} 
		+ \pare{\frac{\pare{1-r}\Dz}{8 \ssi \eps}}^{\frac{1}{1-r}}
	\end{align*}
	for $r \in (0,1)$. For $r \geq 1$ and small enough $\eps$ we have $T_\eps\pare{A^r, \FcDzL} \geq \exp\pare{\nicefrac {\Dz} {\pare{8\ssi\eps}}}$.
\end{theorem}

\begin{proof}
	First note that the definition of $T_\eps$ starts with $\iterationx 0$ instead of $\xfin$. For the sake of consistency with other works, we will also apply this convention in our result by denoting $\pare{\iterationx 0, \xfin, \dots} \gets \pare{\xfin, \iterationx 2, \dots}$. We first consider constant stepsizes $\sst \equiv \ssi \li^{-r}$. Setting $\iterationx 0 = 0$ and applying \Cref{lem:app.optimality.deterministic_arbitrary_ngd_optimiality} yields
	\begin{align*}
		\li^* 
		&= \inf \set{\li \in \N \suchthat \eps 
			> \frac{\Dz - \frac \eps \sm}{2\ssi \li^{1-r}} + \frac{\ssi \sm}{8 \li^r} }\\
		&\geq \inf \set{\li \in \N \suchthat \eps 
			> \frac{\Dz}{4 \ssi \li^{1-r}} + \frac{\Aa}{8\li^r} },
	\end{align*}
	where we used our assumption $\eps \leq \frac{\Dz \sm} 2$ in the last line. For $r \in (0,1)$ we calculate
	\begin{align*}
		\eps > \frac{\Dz}{4 \ssi \li^{1-r}} + \frac{\Aa}{8\li^r}
		&\Rightarrow \eps > \max \set{\frac{\Dz}{4 \ssi \li^{1-r}} , \frac{\Aa}{8\li^r} }\\
		&\Leftrightarrow \li > \max \set{\pare{\frac{\Dz}{4 \ssi \eps}}^{\frac 1 {1-r}}, \pare{\frac{\Aa}{8 \eps}}^{\frac 1 r}}.
	\end{align*}
	In particular we have $\li^* \geq \max \set{\pare{\frac{\Dz}{4 \ssi \eps}}^{\frac 1 {1-r}}, \pare{\frac{\Aa}{8 \eps}}^{\frac 1 r}}$ and hence 
	\begin{align*}
		T_\eps\pare{A_c, \FcDzL} \geq \max \set{\pare{\frac{\Dz}{4 \ssi \eps}}^{\frac 1 {1-r}}, \pare{\frac{\Aa}{8 \eps}}^{\frac 1 r}}.
	\end{align*}
	For $r \geq 1$ we have
	\begin{align*}
		\eps > \frac{\Dz}{4 \ssi \li^{1-r}} + \frac{\Aa}{8\li^r}
		&\Rightarrow \eps > \max\set{\frac{\Dz \li^{r-1}}{4 \ssi}, \frac{\Aa}{8 \li^r}}\\
		&\Leftrightarrow \li > \pare{\frac{\Aa}{8\eps}}^{\frac 1 r} \text{ and } \li^{r-1} > \frac{4 \ssi \eps}{\Dz}
	\end{align*}
	In the case $r=1$ this implies $T_\eps\pare{A_c^r, \FcDzL} \geq \infty$ for all $\eps < \frac{\Dz}{4\ssi}$. For $r>1$ the same holds for $\eps < \pare{\ssi \sm}^\alpha \pare{\frac{\Dz}{\ssi}}^\beta$, where $		\alpha \coloneqq \frac{r-1}{2r - 1}$ and $\beta \coloneqq \frac r {2r-1}$.
	
	Next we consider decreasing stepsizes $\ssi = \ssi \, \ii^{-r}$. Therefore note that, for $q \neq 1$,
	\begin{align*}
		\frac {\ssi  \pare{\pare{\li + 1}^{1-q} - 1}} {1-q}
		= \ssi \int_1^{\li+1} \ii^{-q} d\ii
		\leq  \ssi \iSum \ii^{-q}
		\leq \ssi \pare{1 + \int_1^\li \ii^{-q} d\ii}
		= \ssi \pare{1 + \frac {\li^{1-q} - 1} {1-q} }.
	\end{align*}
	In particular we have 
	\begin{align*}
		2 \eps \iSum \sst &\leq 2\eps \ssi \pare{1 + \frac {\li^{1-r} - 1} {1-r}} \text{ and}\\
		\frac \sm 4 \iSum \sst^2 & \geq \frac{\ssi^2\sm}{4\pare{1-2r}} \pare{\pare{\li + 1}^{1-2r} - 1}.
	\end{align*}
	We first consider $r \in (0,1) \setminus \set{\nicefrac 1 2}$. Plugging into \Cref{lem:app.optimality.deterministic_arbitrary_ngd_optimiality} yields
	\begin{align*}
		T^* 
		&\geq \inf \set{\li \in \N \suchthat \eps > \frac{\Dz\pare{1-r}}{4 \ssi \li^{1-r}} 
			+ \frac{\ssi \sm \pare{1-r}\pare{\li^{1-2r} - 1} }{8 \pare{1-2r} \li^{1-r} } }.
	\end{align*}
	By our assumptions on $\eps$ we get $\li^* \geq 2$ and hence can assume $\li \geq 2$ which in turn implies $\frac{\li^{1-2r} - 1}{1-2r} \geq \frac{\li^{1-2r}}{2}$. Therefore we get
	\begin{align*}
		\li^*
		&\geq \inf \set{\li \in \N \suchthat \eps > \frac{\Dz\pare{1-r}}{4 \ssi \li^{1-r}} 
			+ \frac{\ssi \sm \pare{1-r}}{16 \li^r } }\\
		&\geq  \max\set{
			\pare{\frac{\Dz\pare{1-r}}{4\ssi\eps}}^{\frac 1 {1-r}},
			\pare{\frac{\ssi \sm \pare{1-r}}{16 \eps}}^{\frac 1 r}}.
	\end{align*}
	Finally we have to consider the edge cases $r \in \set{\nicefrac 1 2, 1}$. Let $r = \nicefrac 1 2$, then 
	\begin{align*}
		\iSum \sst &\leq \ssi \pare{2 \sqrt \li - 1} \leq 2\ssi\sqrt\li,\\
		\iSum \sst^2 &\geq \ssi^2 \int_1^{\li+1} \ii^{-1} d\ii
		\geq \ssi^2 \log \pare{\li}
	\end{align*}
	and hence
	\begin{align*}
		\li^* 
		&\geq \inf \set{\li \in \N \suchthat \eps > \frac{\Dz}{8 \ssi \sqrt \li} + \frac{\ssi \sm \log \pare \li}{16 \sqrt \li}}\\
		&\geq \inf \set{\li \in \N \suchthat \eps > \frac{\Dz \pare{1-r}}{4 \ssi \li^r} + \frac{\ssi \sm \pare{1-r} }{16 \li^r}}\\
	\end{align*}
	and we can hence proceed as before. Note that a more careful analysis can additionally show tightness of the $\log \pare \li$ dependence.
	Now let $r = 1$ and note that 
	\begin{align*}
		\iSum \sst = \ssi \iSum \ii^{-1} \leq \ssi \pare{1 + \log \pare \li}.
	\end{align*}
	In particular
	\begin{align*}
		\li^* 
		&\geq \inf \set{\li \in \N \suchthat \eps > \frac{\Dz}{4 \ssi \pare{1 + \log \pare \li}}}
		\geq \exp\pare{\frac{\Dz}{4 \ssi \eps} - 1}
	\end{align*}
	and hence $T_\eps\pare{A_d^r, \FcDzL} \geq e^{\frac{\Dz}{8\ssi \eps}}$ for $\eps \leq \frac{\Dz}{8\ssi}$.
\end{proof}

\subsection{Stochastic Setting}
\label{sec:app.optimality.stoch}

Finally we will extend the above result to the stochastic setting.

\begin{theorem}\label{thm:app.optimality.stochastic}
	Let $\FcDzL$ be the set of functions that satisfy \lb\ and \lsmooth. Furthermore let $\Oc_{\varsym, p}$ denote the set of stochastic gradient oracles that satisfy \pBCM. Suppose $\eps \leq \frac{\Dz \sm}{2}$ and let $\ssi, B, q > 0, r \in (0,1)$. Let $\Ac$ denote \algname{NGD}\ with parameters $\sst \equiv \ssi \li^{-r}, \Bt \equiv \max\set{1,B\li^q}$ and the mini-batch gradient estimator $\gt = \frac 1 {\Bt} \sum_{j = 1}^{\Bt} \nf{\xt,\xit^{(j)}}$, where $\xit^{(1)}, \dots, \xit^{(\Bt)} \iid \xit$. Then there exists a function $F \in \FcDzL$ and oracle $\nabla f(\cdot, \cdot) \in \Oc_{\varsym,p}$ such that $\Ac$ requires at least
	\begin{align*}
		m_\eps^\E \geq 
		\max \set{
			\pare{\frac{\Dz}{6 \ssi \eps}}^{\frac 1 {1-r}}, \pare{\frac{\Aa}{12 \eps}}^{\frac 1 r},
			B\pare{\frac{\Dz}{6 \ssi \eps}}^{\frac {1+q} {1-r}}, B\pare{\frac{\Aa}{12 \eps}}^{\frac {1+q} r},
			\frac 1 {B^{\frac 1 q}} \pare{\frac{\varsym}{28 \eps}}^{\frac{p \pare{q + 1}}{q\pare{p-1}}}.
		}
	\end{align*}
	oracle calls to generate an iterate with $\Expnorm{\nFxt} \leq \eps$.
\end{theorem}

When comparing to the corresponding upper bound \eqref{eq:app.missing.exp.param_free.sample_complexity} we can see that both bounds are tight in all parameters. This is due to $\frac 1 n \sum_{i=1}^n a_i \leq \max\set{a_1, \dots, a_n} \leq \sum_{i=1}^n a_i$, hence the maximum and sum notation are equivalent up to constants.

\begin{proof}
	\newcommand{\opr}{\pare{1+\rho}}
	\newcommand{\ssl}{\bar{\ssi}}
	\newcommand{\Bl}{\bar{B}}
	\newcommand{\detLb}{\li^*_d}
	\newcommand{\stochLb}{\li^*_s}
	The idea behind the proof is the following. We will again use a very similar construction to \eqref{eq:app.optimality.deterministic_arbitrary_ngd_optimiality.derivative_construction} and add a noise oracle on top. The goal of this noise oracle will be to point in the wrong direction with the highest possible probability, effectively slowing down the progress we make even more. As this noise oracle may however lead to iterates going below $\xfin$, we need to slightly modify the construction of $F$. 
	
	\paragraph{Construction of the hard function $F$.} To this end, let $\ssl \coloneqq \ssi \li^{-r}$, $\tau_k \coloneqq k \ssl$ for $k \in \Z$, 
	\begin{align*}
		\detLb 
		&\coloneqq \inf \set{\li \in \N \suchthat \eps > \frac{\Dz - \frac \eps \sm}{3 \li \ssl} + \frac{\sm \ssl}{12}}\\
		&\geq \max \set{\pare{\frac{\Dz}{6 \ssi \eps}}^{\frac 1 {1-r}}, \pare{\frac{\Aa}{12 \eps}}^{\frac 1 r}}.
	\end{align*}
	where we used $\eps \leq \frac {\Dz \sm} 2$ in the last inequality as before, and define
	\begin{align}\label{eq:app.optimality.stochastic.derivative_construction}
		F'(x) &\coloneqq
		\begin{cases}
			g_{\ssl} \pare{x-\tau_\ii}, & x \in [\tau_\ii, \tau_{\ii+1}) ,\, \ii + 1 \leq \detLb\\
			-2\eps + L\pare{x-\iterationtau {\detLb}}, & x \in \left(\iterationtau {\detLb}, \iterationtau {\detLb*} + \frac{3\eps}{L}\right]\\
			0, & \text{otherwise},
		\end{cases}
	\end{align}
	where
	\begin{align*}
		g_{\ssi} &\colon \left[0, \ssi \right] \to \R \quad \text{ is given by } \quad x \mapsto 
		\begin{cases}
			-3\eps + \sm x, & x \leq \frac {\ssi} 2\\
			-3\eps + {\ssi \sm} - \sm x, & x > \frac \ssi 2.
		\end{cases}
	\end{align*}
	As before, we define the hard function as $F(x) \coloneqq \Dz + \int_0^x F'(t) d t$. In the following we will denote the derivative of $F$ using $\nF$ to align with the gradient oracle notation. Now firstly note that we can use the deterministic \Cref{lem:app.optimality.deterministic_arbitrary_ngd_optimiality} to rule out any stepsize that satisfies $\ssl \geq \frac{8 \eps}{\sm}$: In this case we would have
	\begin{align*}
		\frac{\Dz - \frac \eps \sm}{2\li\ssl} + \frac{\sm \li \ssl^2}{8 \li \ssl}
		= \frac{\Dz - \frac \eps \sm}{2\li\ssl} + \frac{\sm\ssl}{8}
		\geq 0 + \eps,
	\end{align*}
	where we used $\eps \leq \frac{\Dz \sm}{2}$ in the last inequality. Hence we may assume $\ssi \li^{-r} \leq \frac{8\eps}{\sm}$. Under this assumption\footnote{Note that for $\ssi \li^{-r} \geq \frac {12\eps} \sm$ we would get $\lim_{x\to-\infty} F(x) = -\infty$ for $F$ defined by \eqref{eq:app.optimality.stochastic.derivative_construction}.} we again have $F(0) - \inf_x F(x) \leq \Dz$ and that $F$ is $L$-smooth.
	\newcommand{\nFx}{\nF \pare x}
	\newcommand{\nnFx}{\norm{\nFx}}
	\newcommand{\nsqnFx}{\nnFx^2}
	{
	\paragraph{Construction of the noise oracle $\nf{x,\xi}$.}
	We will construct the mentioned oracle, which aims to point in the wrong direction with the maximal probability. This oracle construction follows a similar idea as \citep[Theorem 3]{yang2023two_sides}. To construct the oracle, let $\rho > 0$ to be defined later and define $\alpha \coloneqq \frac p {p-1}$,
	\newcommand{\dx}{\delta \pare{x}}
	\newcommand{\omdx}{\pare{1-\dx}}
	\begin{align*}
		\dx \coloneqq \min\set{1,\pare{\frac{2 \opr \norm{\nF \pare x}}{\varsym}}^\al}.
	\end{align*}
	This will be the probability of the oracle returning \emph{the correct} direction. Now let
	\begin{align*}
		\nf{x, \xi} \coloneqq 
		\begin{cases}
			-\rho \nF(x), & \xi \geq \dx\\
			\pare{1 + \frac{\pare{1-\dx} \opr} {\dx}}\nF(x),& \xi < \dx
		\end{cases}
	\end{align*}
	and $\xi \sim \unif\pare{[0,1]}$. Straightforward calculations yield 
	\begin{align*}
		\Exp{\nf{x,\xi}} 
		&= \nF\pare x \pare{\omdx(- \rho) + \dx \pare{1 + \frac{\omdx \opr} {\delta \pare x}}}\\
		&= \nF \pare x \pare{ \omdx(- \rho) +\dx + \omdx \opr }\\
		&= \nF \pare x 
	\end{align*}
	and 
	\begin{align*}
		\Exp{\norm{\nf{x, \xi} - \nF \pare x}^p}
		&= \omdx \opr^p \nnFx^p + \delta(x) \pare{\frac{\omdx\opr}{\dx} \nnFx }^p\\
		&= \opr^p \nnFx^p \pare{\omdx + \frac{\omdx^p}{\dx^{p-1}}}\\
		&\leq \opr^p \nnFx^p \pare{1 + \frac 1 {\dx^{p-1}}}\\
		&\leq \frac{2\opr^p\nnFx^p}{\dx^{p-1}} \leq \varsym^p .
	\end{align*}
	In particular $\nf{\cdot, \cdot} \in \Oc_{\sigma,p}$.}
	\newcommand{\sepsa}{\pare{7\eps}^\al}
	\paragraph{The behaviour of \nsgdName on the constructed function and oracle.} Finally we are able to show the lower bound by analysing the behaviour of \nsgdName on our constructed objects. 
	Firstly it is clear that we can upper bound the stochastic progress by the deterministic progress, i.e.\ $\xtp \leq \ssl \ii$. In particular, with the same arguments as in \Cref{lem:app.optimality.deterministic_arbitrary_ngd_optimiality,thm:app.optimality.deterministic_polynomial_ngd_optimiality}, we get that $\nnFxt > \eps$ for all $\ii \in \left[\detLb\right]$ and we hence can lower bound the iteration complexity by $\li \geq \detLb$. We next differentiate two cases.\\
	\textbf{Case 1: {$\max\set{1,B\pare{\detLb}^q} \geq \frac{\varsym^\al}{2\sepsa}$.}} In this case nothing else needs to be done and we can lower bound the number of oracle calls required to find an $\eps$-stationary point by
	\begin{align*}
		\detLb \cdot \max\set{1,B\pare{\detLb}^q}
		&= \max\set{\detLb, B \pare{\frac{\Dz}{6\ssi\eps}}^{\frac{1+q}{1-r}}, B \pare{\frac{\ssi\sm}{12\eps}}^{\frac{1+q}{r}} }.
	\end{align*}
	
	\textbf{Case 2: {$\max\set{1,B\pare{\detLb}^q} < \frac{\varsym^\al}{2\sepsa}$}.} In this case we will make use of the gradient oracle to construct a stronger lower bound $\stochLb > \detLb$. Therefore first note that, due to the constant stepsize and $\xfin = 0$, the iterations $\xt$ always stay on the lattice $\Gamma = \ssl \Z$. Furthermore, by the construction of $F$, we have that
	\begin{align*}
		\fas x \in \Gamma\colon \nFx \in [-3\eps, 0],
	\end{align*}
	which{, for $\rho \leq \frac 1 6$,} in particular implies
	\begin{align}\label{eq:app.optimality.stochastic.delta_bound}
		\fas x \in \Gamma \colon { \delta(x) \leq \pare{\frac{6 \opr\eps}{\varsym}}^\al \leq \pare{\frac{7\eps}{\varsym}}^\al } \eqqcolon \delta.
	\end{align}
	\newcommand{\ind}{1}
	Now define the random variable $\zetat \coloneqq 1_{\set{\gt < 0}}$, where $\ind_A(\omega) = \begin{cases} 1, & \omega \in A\\ 0, & \text{o.w.}\end{cases}$, and compute
	\begin{align}\label{eq:app.optimality.stochastic.x_t}
		\xtp = \xt + \ssl \pare{2 \zetat - 1} 
		= \ssl \sum_{\tau = 1}^\ii \pare{2 \iterationzeta \tau - 1}
		\eqqcolon 2\ssl \St - \ssl \ii,
	\end{align}
	where $\St \coloneqq \sum_{\tau = 1}^\ii \iterationzeta \tau$. Furthermore, let $\Bl = \max\set{1,B\li^q}$ and $\rho \leq \min\set{{\nicefrac 1 6}, \nicefrac 1 {\Bl}}$, then we have
	\begin{align}\label{eq:app.optimality.stochastic.prob_bound_temp}
		\P \pare{\gt < 0}
		&= \P \pare{\frac 1 {\Bl}  \sum_{j = 1}^{\Bl} \nf{\xt, \xit^{(j)}} < 0} 
		= 1 - \pare{1-\delta\atxt}^{\Bl} 
		\stackrel{\eqref{eq:app.optimality.stochastic.delta_bound}}{\leq} 1 - \pare{1-\delta}^{\Bl}.
	\end{align}
	By definition $\pare{1-\delta}^{\Bl} = \exp\pare{\Bl \log \pare{1-\delta}}$ and
	\begin{align*}
		{\log\pare{1-\delta}
		\geq \frac{-\delta}{1 - \delta}
		= - \frac{\sepsa}{\varsym^\al - \sepsa}
		\geq - \frac{2\sepsa}{\varsym^\al}
		},
	\end{align*}
	where we used {$\log\pare{1+x} \geq \frac{x}{1+x}$ for $x \in (-1,0]$ in the first, and $1 <  \frac{\varsym^\al}{2\sepsa}$ by our case 2 assumption} in the last inequality.
	Plugging into \eqref{eq:app.optimality.stochastic.prob_bound_temp} yields
	\begin{align}\label{eq:app.optimality.stochastic.prob_bound}
		\P \pare{\gt < 0}
		\leq 1 - \exp\pare{- \Bl {\frac{2\sepsa}{\varsym^\al}}}
		\leq \min\set{1, \Bl {\frac{2\sepsa}{\varsym^\al}}},
	\end{align}
	where we used $e^x \geq \max\set{0, 1 + x}$ in the last inequality. Next up let $\li \in \N$ such that $\Bl \leq {\frac{\varsym^\al}{4 \sepsa}}$, then
	\begin{align*}
		\P\pare{\gt < 0}
		\stackrel{\eqref{eq:app.optimality.stochastic.prob_bound}}{\leq}
		\min\set{1, \Bl {\frac{2\sepsa}{\varsym^\al}}}
		\leq \frac 1 2
	\end{align*}
	and $\zetat$ is hence a Bernoulli random variable with probability at most $\nicefrac 1 2$. In particular, we have $\median\pare{\St}\leq \floor{\nicefrac \ii 2}$ and hence
	
	\begin{align*}
		\frac 1 2 
		\leq \P\pare{\St \leq \frac \ii 2} 
		= \P \pare{2\ssl \St \leq \ssl \ii} 
		\stackrel{\eqref{eq:app.optimality.stochastic.x_t}}{=} \P \pare{\xtp \leq 0}.
	\end{align*}
	Finally note that all $x$ in $\pare{\Gamma \cap (-\infty, 0]}$ satisfy $\nF(x) = - 3\eps$ and hence
	\begin{align*}
		\Exp{\norm{\nF\atxt}}
		= 3 \eps \P \pare{\xt \leq 0} + \Exp{\norm{\nF\atxt} 1_{\set{\xt > 0}}}
		> \eps + 0
	\end{align*}
	for all $\ii \in [\li]$. Summing up the results in \emph{Case 2}, we so far proved the auxiliary result that for any $\li$ with $\max\set{1,B\li^q} \leq {\frac{\varsym^\al}{4 \sepsa}}$ all iterates $\ii \in [\li]$ satisfy $\Exp{\norm{\nF\atxt}} > \eps$. By the assumption of case 2, this implies
	\begin{align*}
		\fas \li \leq \stochLb\colon \fas \ii \in [\li] \colon \Exp{\norm{\nF\atxt}} > \eps,
	\end{align*}
	where $\stochLb \coloneqq \pare{{\frac{\varsym^\al}{4 B\sepsa}}}^{\nicefrac 1 q} > \detLb$ with $\alpha = \frac p {p-1}$. In particular, we can lower bound the number of oracle calls required to reach an expected $\eps$-stationary point by
	\begin{align*}
		\stochLb \cdot B \pare{\stochLb}^q 
		= {\frac 1 {B^{\frac 1 q}} \pare{\frac{\varsym}{28 \eps}}^{\frac{p \pare{q + 1}}{q\pare{p-1}}} }.
	\end{align*}
	
	\paragraph{Combining.} Finally we are able to combine everything into our lower bound. Therefore, let
	\begin{align*}
		\li^* \coloneqq \max \set{
			\detLb, 
			B \pare{\frac{\Dz}{6\ssi\eps}}^{\frac{1+q}{1-r}}, 
			B \pare{\frac{\ssi\sm}{12\eps}}^{\frac{1+q}{r}}, 
			{\frac 1 {B^{\frac 1 q}} \pare{\frac{\varsym}{28 \eps}}^{\frac{p \pare{q + 1}}{q\pare{p-1}}} }}
	\end{align*}
	and note that for $\li \leq \li^*$, one of the above cases applies, showing
	\begin{align*}
		\fas \ii \in [\li] \colon \Exp{\norm{\nF\atxt}} > \eps.
	\end{align*}
	This completes the proof.
\end{proof}

\subsection{Lower-Bound on the Convergence Measure}
\label{sec:app.missing.lb}

\begin{proof}
	In this proof, we consider a slightly more general step-size $\gamma = \sqrt{\frac{\Delta_1}{L}} \frac{1}{T^a}$ for any $a < 1$. Take $F(x) = \frac{1}{2} x^2 $ and $x_1 > 0$ (w.l.g.), then the step-size is $\gamma = \frac{x_0}{T^a}$. Denote by $N := \sqrt{2} T^a $. For the first $\ceil{N}$ iterations the update rule is $x_t = x_1 - \gamma (t-1) = x_1 (1 + \frac{1}{N} - \frac{t}{N} ) \geq 0$, for $t = 0, \ldots, \ceil{N}$. We compute for $T \geq \ceil{N} = \ceil{\sqrt{2} T^a} $
	\begin{eqnarray*}
		\sum_{t=1}^{T} \sqnorm{\nabla F(x_t)} &\geq& \sum_{t=1}^{\ceil{N}} \sqnorm{\nabla F(x_t)} = x_1^2 \sum_{t=1}^{\ceil{N}} \rb{ 1 + \frac{1}{N} - \frac{t}{N} }^2 
		\\
		&=& x_1^2 \ceil{N} \rb{1 + \frac{1}{N} }^2 - 2 x_1^2 \rb{1 + \frac{1}{N} } \frac{1}{N} \sum_{t=1}^{\ceil{N}} t + \frac{x_1^2}{N^2} \sum_{t=1}^{\ceil{N}} t^2 \\
		&=& x_1^2 \ceil{N} \rb{1 + \frac{1}{N} }^2 - 2 x_1^2 \rb{1 + \frac{1}{N} } \frac{1}{N} \frac{\ceil{N} ( \ceil{N} + 1 )}{2}  \\
		&&\qquad + \frac{x_1^2}{N^2} \rb{\frac{\ceil{N}^3}{3} + \frac{\ceil{N}^2}{2} + \frac{\ceil{N}}{6} } \\
		&\geq& x_1^2 \frac{\ceil{N} (N+1) (N - \ceil{N}) }{N^2} + \frac{ x_1^2 \ceil{N}^3 }{ 3 \, N^2 }   \\
		&\geq& -  \frac{x_1^2\ceil{N}^2 }{N^2} + \frac{ x_1^2 \ceil{N}^3 }{ 3 \, N^2 }   \\
		&\geq&  \frac{ x_1^2 \ceil{N}^3 }{ 6  \, N^2 } \rb{2 - \frac{6}{\ceil{N}}}  \geq \frac{ x_1^2 \ceil{N}^3 }{ 6  \, N^2 } \geq \frac{ x_1^2 N }{ 6 } =  \frac{ \sqrt{2} L \Delta_1 T^a }{3 } , 
	\end{eqnarray*}
	where in the second inequality we dropped the last two terms, in the third inequality we used $ \ceil{N} - N \leq 1$ and $N+1 \leq \ceil{N}^2$ for $N \geq 6$. The forth inequality holds by the assumption $N \geq 6$. It remains to divide both sides by $T$ and verify that in case $a = 1/2$, the assumption $T\geq 18$ implies the assumed conditions $T \geq \ceil{N} \geq 6$. Rearranging, we get
	$$
	\sqrt{ \Exp{ \sqnorm{\nabla F(\bar x_T)} } } \geq 
	\sqrt{ \frac{ \sqrt{2} L \Delta_1}{ 3 T } } \cdot T^{1/4} .
	$$
	Noting that $\sqrt{\frac{\sqrt 2}{3}} \geq \frac 2 3$ completes the proof.
\end{proof}

\clearpage

\section{ADDITIONAL EXPERIMENTS AND DETAILS}
\label{sec:app.add_experiments}
In this section we provide additional information and experiments for \Cref{sec:experiments}.

\subsection{Additional Details on Language Modelling}
\label{sec:app.add_experiments.lstm}
 
\paragraph{Additional Details}
All experiments were carried out on Nvidia RTX 3090 GPUs in an internal cluster. The total compute including preliminary experiments were approximately 380 GPU hours. Roughly 200 of these were required for preliminary experiments and parameter-tuning, 180 for the final experiments.

The AWD-LSTM \citep{LSTM2018Merity} is released under a BSD 3-Clause License, the Penn Treebank dataset \citep{PTB1993Marcus} under the LDC User Agreement for Non-Members and the WikiText-2 dataset \citep{Wikitext2017Merity} under the Creative Commons BY-SA 3.0 license.

The below experiments all follow the general structure outlined in \Cref{sec:experiments}.

\paragraph{Reasons for the Clipping Behaviour.} We additionally want to understand why the percentage of clipped gradients increases over time. Therefore \Cref{fig:app.additional_experiments.minibatch_grad_norm} examine the average minibatch-gradient norm per epoch, i.e.\ $\frac 1 E \sum_{\ii = \ii_0}^{\ii_0 + E - 1} \norm{\gt}$ where the epoch consists of $E$ mini-batches and starts at iteration $\ii_0$. The plot shows that, while the training loss decreases, the stochastic gradient norms increase. This observation is in line with previous observations on different tasks \citep[Chapter 8]{DeepLearning2016Goodfellow}. Therefore, while surprising at first, the increase in stochastic gradient norms is able to explain the increasing clipping percentage in hindsight.

\begin{figure*}[!ht]
	\centering
	\hfill
	\begin{subfigure}[c]{0.45\linewidth}
		\includegraphics[width=\linewidth]{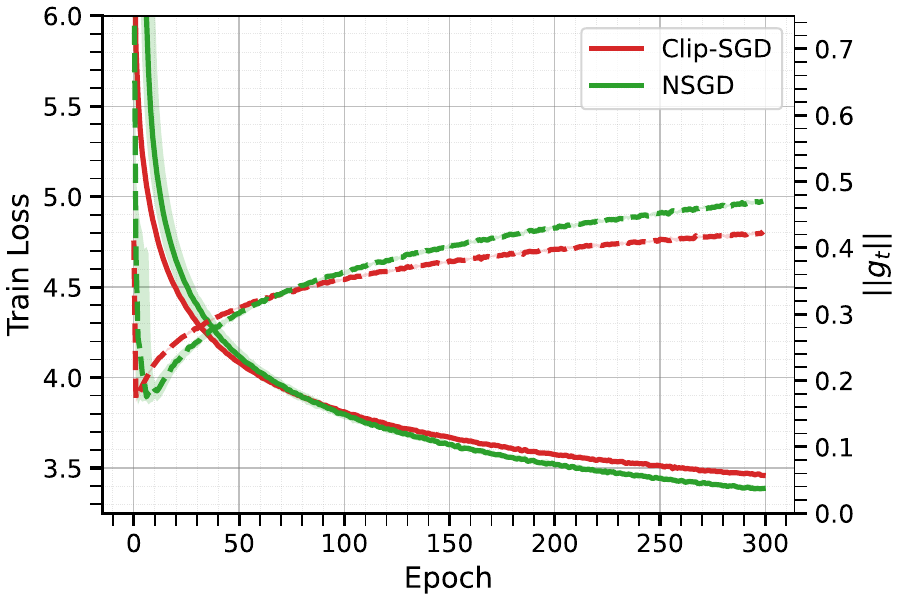}
		\caption{PTB}
		\label{fig:app.add_experiments.minibatch_grad_norm.ptb}
	\end{subfigure}
	\hfill
	\begin{subfigure}[c]{0.45\linewidth}
		\includegraphics[width=\linewidth]{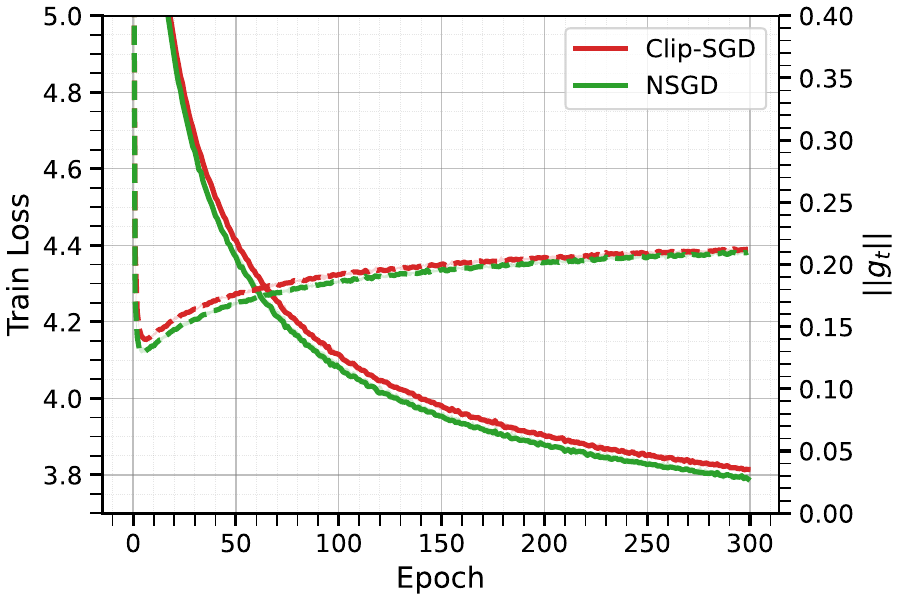}
		\caption{WikiText-2}
		\label{fig:app.add_experiments.minibatch_grad_norm.wikitext}
	\end{subfigure}
	\hfill
	\caption{All plots consider \clippedName\ and \nsgd\ with tuned parameters. Solid lines represent the training loss and correspond to the left y-axis. The dashed line corresponds to the right y-axis, and represents the average mini-batch gradient norm in an epoch. Shaded areas represent the minimal and maximal value within 5 seeds, the line the median.}
	\label{fig:app.additional_experiments.minibatch_grad_norm}
\end{figure*}

\subsection{Illustrating Drawbacks of Gradient Clipping Theory}
In the two sets of experiments below we compare several algorithms with and without parameter tuning on a simple quadratic model to better understand the influence of heavy tailed noise and the necessity of parameter tuning.

\paragraph{Comparison with tuned parameters.} In this set of experiments, we compare \nsgd\ using step-sizes $\sst =  \eta / {\sqrt \ii}$ \citep{yang2023two_sides} and \clippedSGD\ with $\sst =  \eta / {\sqrt \ii}, \, \clipt = \clipThresh \cdot \sqrt[4] \ii$ \citep{zhang2020adaptive,nguyen2023improved} and tune the pair $(\eta, \gamma)$ over a grid. We report the optimal parameters, $\eta$ for \nsgd\ and the pair $(\eta, \gamma)$ for \clippedSGD\ in \Cref{tab:comparison}. Convergence plots for different noise distributions are presented in \Cref{fig:intro.sgd_clip_nsgdm_comparison_Pareto_1_5,fig:intro.sgd_clip_nsgdm_comparison,fig:intro.sgd_clip_nsgdm_comparison_light}.   

We observe that in both heavy tailed and light tailed settings, when parameters are tuned, the two algorithms exhibit comparable convergence rate. However, for all three noise distributions we tested, the optimal parameter for \nsgd\ appeared to be $\eta = 0.5$, while \clippedSGD\ required different parameters to reach similar performance. This illustrates that \clippedSGD\ can be more sensitive to misspecification of its parameters compared to \nsgd. Moreover, \clippedSGD\ requires two parameters for tuning compared to only one parameter for \nsgd.

\begin{table}[ht]
\centering
\caption{Comparison of Tuned Parameters for \nsgd\ and \clippedSGD\ under Different Noise Distributions}
\begin{tabular}{|c|c|c|c|}
\hline
\textbf{Noise Distribution} & \textbf{Algorithm} & \textbf{Optimal} $\eta$ & \textbf{Optimal} $\gamma$ \\ \hline
Heavy tailed ($p=1.5$) & \nsgd\ & 0.5 & -- \\ 
(\Cref{fig:intro.sgd_clip_nsgdm_comparison_Pareto_1_5})  & \clippedSGD\ & 0.1 & 1 \\ \hline
Heavy tailed ($p=1.8$) & \nsgd\ & 0.5 & -- \\ 
(\Cref{fig:intro.sgd_clip_nsgdm_comparison}) & \clippedSGD\ & 100 & 0.001 \\ \hline
Light tailed & \nsgd\ & 0.5 & -- \\ 
(\Cref{fig:intro.sgd_clip_nsgdm_comparison_light}) & \clippedSGD\ & 5 & 0.1 \\ \hline
\end{tabular}
\label{tab:comparison}
\end{table}

\paragraph{Comparison without parameter tuning.} In \Cref{fig:intro.quadraticheavy_Pareto_1_5,fig:intro.quadraticheavy,fig:intro.quadraticlight}, we compare several adaptive algorithms with default parameter sequences, i.e., $\gamma = \eta = 1$. Specifically, we use $\eta_t = 1/\sqrt{t}$ for \sgd and \nsgd; $\eta_t = \sqrt{\alpha_t / t}$, $\alpha_t = 1/\sqrt{t}$ (where $\alpha_t$ is momentum sequence) \citep{yang2023two_sides}, and $\sst =  1 / {\sqrt \ii}, \, \clipt = \sqrt[4] \ii$ for \clippedSGD. This order of step-sizes is selected based on the theory for each algorithm under BV setting, where this order is known to give an asymptotically optimal convergence rate as $T\rightarrow \infty$. We observe that the performance of \clippedName significantly degrades when $\gamma$ and $\eta$ are not tuned, while the performance of \nsgd remains nearly the same. 

We also see that under heavier tailed noise such as Pareto distribution, the performance of untuned \sgd and \clippedName degrades substantially compared to the light tail noise setting, confirming the sensitivity of these algorithms to different noise distributions. On the other hand, \nsgd\ and its momentum variant (see \eqref{eq:app.nsgdm} in \Cref{sec:app.nsgdm}) are more stable and converge to a smaller neighbourhood around the optimal solution even under heavy tailed noise.

\begin{figure}[t!]
	\centering
	\hspace*{\fill}
    \subcaptionbox{\label{fig:intro.sgd_clip_nsgdm_comparison_Pareto_1_5}%
    Tuned \clippedName vs. \nsgd. When parameters are tuned, clipping is triggered at every iteration.}%
	[.45\textwidth]%
	{\includegraphics[width=\linewidth]{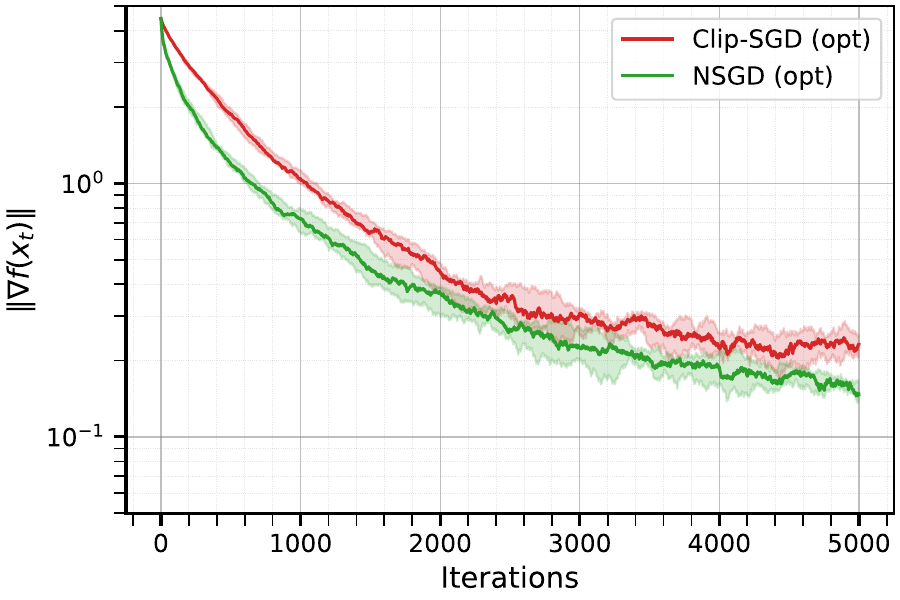}}
    \subcaptionbox{\label{fig:intro.quadraticheavy_Pareto_1_5}%
	Comparison \textit{without parameter tuning}. }%
	[.45\textwidth]%
	{\includegraphics[width=\linewidth]{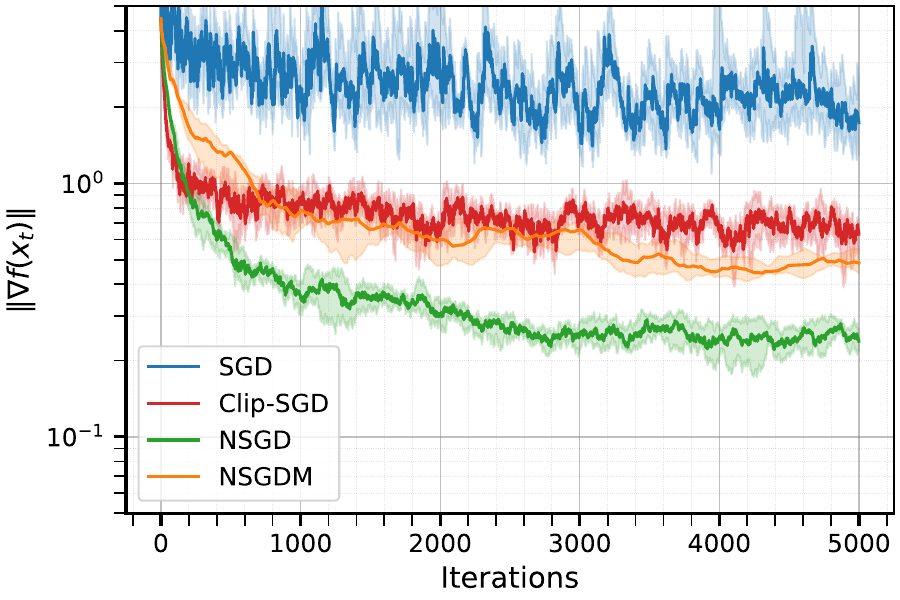}}
	\hspace*{\fill}
	\caption{Performance of different algorithms on a quadratic problem $f(x, \xi) = \frac{1}{2}\sqnorm{x}_2 + \langle x, \xi \rangle$, $d = 10$, where $\xi$ is a random vector with i.i.d.\,components drawn from a symmetrized Pareto distribution with tail index $p = 1.5$.}
	\label{fig:appendix.motivation_heavy_1_5}
\end{figure}

\begin{figure}[t!]
	\centering
	\hspace*{\fill}
    \subcaptionbox{\label{fig:intro.sgd_clip_nsgdm_comparison}%
    Tuned \clippedName vs. \nsgd. When parameters are tuned, clipping is triggered at every iteration.}%
	[.45\textwidth]%
	{\includegraphics[width=\linewidth]{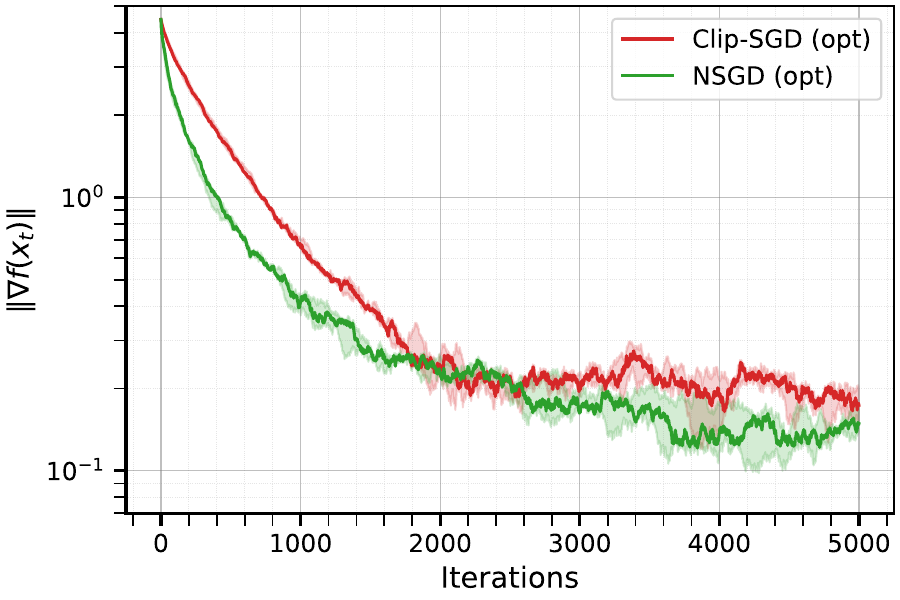}}
    \subcaptionbox{\label{fig:intro.quadraticheavy}%
	Comparison \textit{without parameter tuning}. }%
	[.45\textwidth]%
	{\includegraphics[width=\linewidth]{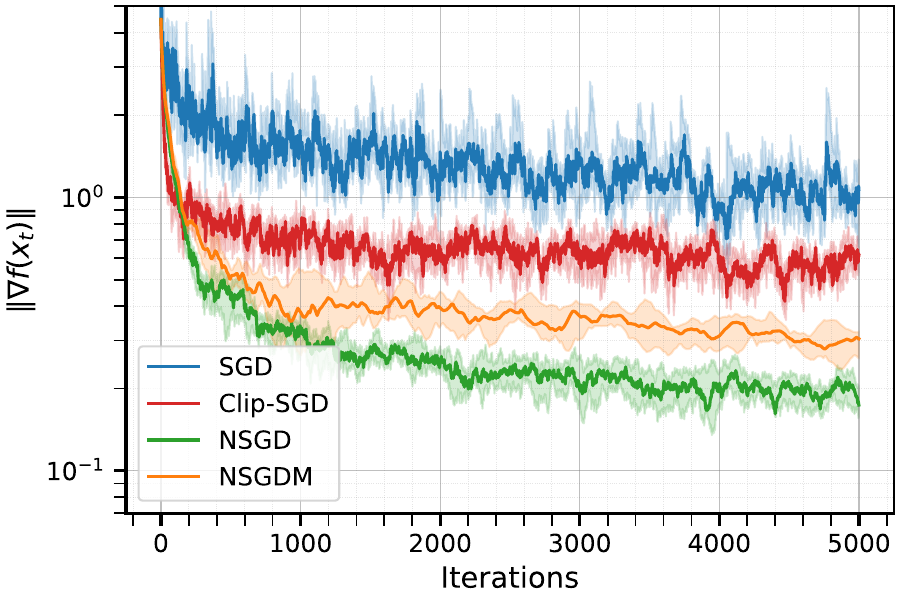}}
	\hspace*{\fill}
	\caption{Performance of different algorithms on a quadratic problem $f(x, \xi) = \frac{1}{2}\sqnorm{x}_2 + \langle x, \xi \rangle$, $d = 10$, where $\xi$ is a random vector with i.i.d.\,components drawn from a symmetrized Pareto distribution with tail index $p = 1.8$.}
	\label{fig:appendix.motivation_heavy}
\end{figure}

\begin{figure}[t!]
	\centering
	\hspace*{\fill}
    \subcaptionbox{\label{fig:intro.sgd_clip_nsgdm_comparison_light}%
    Tuned \clippedName vs. \nsgd. When parameters are tuned, clipping is triggered at every iteration.}%
	[.45\textwidth]%
	{\includegraphics[width=\linewidth]{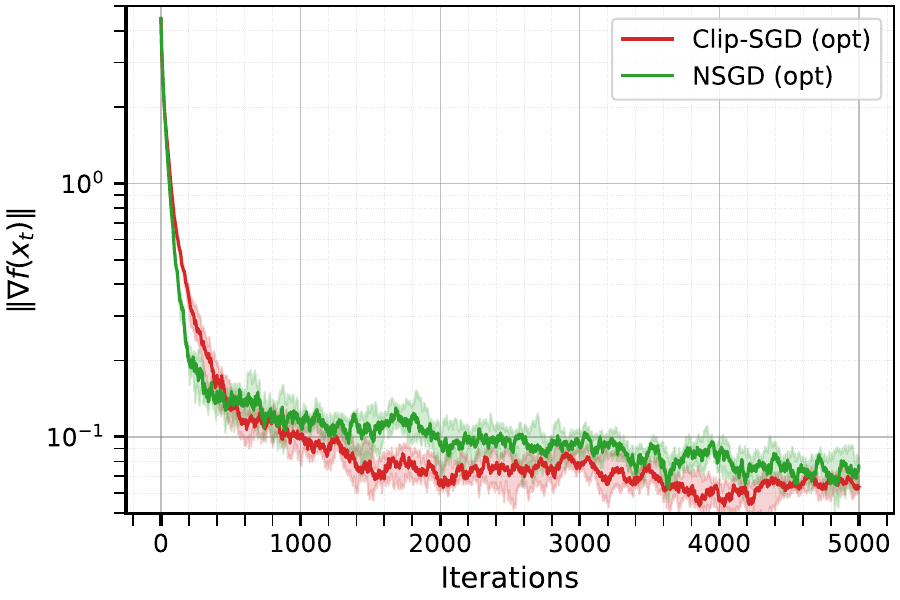}}
    \subcaptionbox{\label{fig:intro.quadraticlight}%
	Comparison \textit{without parameter tuning}. }%
	[.45\textwidth]%
	{\includegraphics[width=\linewidth]{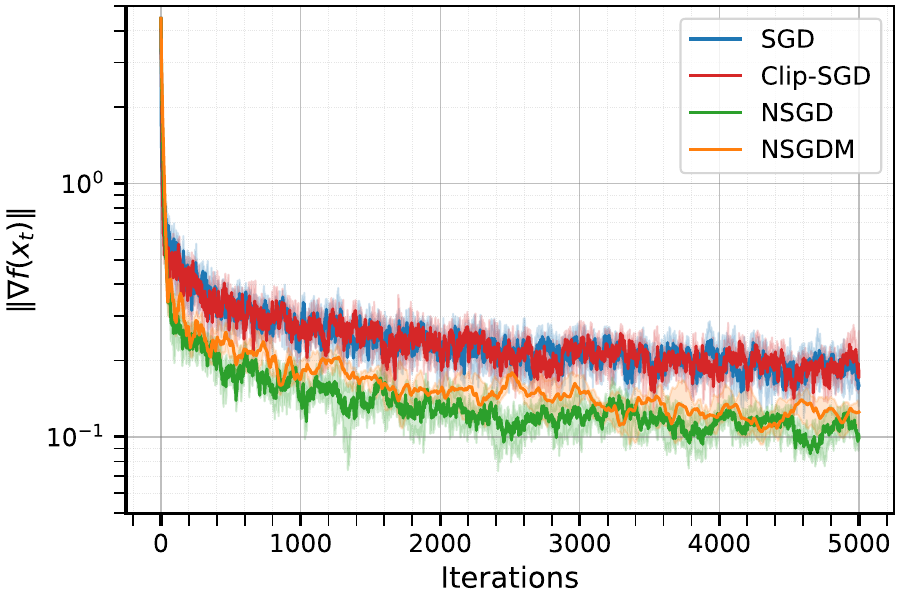}}
	\hspace*{\fill}
	\caption{Performance of different algorithms on a quadratic problem $f(x, \xi) = \frac{1}{2}\sqnorm{x}_2 + \langle x, \xi \rangle$, $d = 10$, where $\xi$ is distributed according standard Normal distribution.}
	\label{fig:appendix.motivation_light}
\end{figure}

\subsection{Verifying High Probability Bounds for SGD and NSGD(M)}
\label{sec:appendix:verify_hp}

In this section, we conduct experiments to verify high probability convergence for three algorithms: \sgdName, \nsgdName, \nsgdm. High probability convergence refers to the convergence rate of the average gradient norm depending linearly on $\log(1/\delta)$ as demonstrated in our \Cref{cor:main.optimal_hp}, where $\delta \in (0,1)$ is a failure probability. Previous theoretical results have shown that vanilla \sgdName does not exhibit high probability convergence. In particular, \citet{sadiev2023high} demonstrated that under a bounded variance setting, \sgdName fails to achieve this property, mainly due to noise injected in the final iteration. In contrast, for adaptive methods such as \clippedName \citep{nguyen2023improved} or \nsgd\ \Cref{cor:main.optimal_hp}, one can establish a high probability convergence with a mild linear dependence on $\log(1/\delta)$. However, extending our result to \nsgdm \space is challenging due to correlation issues introduced by momentum. Thus, the primary objective of this section is to empirically investigate whether \nsgdm \space might still exhibit high probability convergence similar to \nsgdName.

To achieve this, we evaluate the performance of the three algorithms on a simple quadratic function $F(x) = \frac{1}{2} \|x\|^2$,  $x\in \R^d$ using dimensions $d = 1$ and $d = 1000$. We introduce three types of noise during training: (standard) Normal, (component-wise) Pareto with $p = 1.5$ and $p = 2.5$, to simulate both light-tailed and heavy-tailed noise environments. Each algorithm is run $k = 10^5$ times over $T = 100$ iterations, and the convergence criterion is the average gradient norm over $T$ iterations. We present two plots for each set of experiments. The left plot visualizes the convergence behavior by selecting the median, $\delta$ and $1-\delta$ quantiles (where $\delta := 10^{-4}$) of the algorithm runs based on the average gradient norm at $T = 100$. These quantile-based trajectories are plotted against iteration $t = 1, \ldots, T$. The right plot shows the quantiles of  average gradient norm at $T = 100$ plotted against $\log(1/\delta)$. For algorithms with high probability convergence, this plot should have a sub-linear dependence on $\log(1/\delta)$.

\begin{figure}
	\centering
	\includegraphics[width=0.9\linewidth]{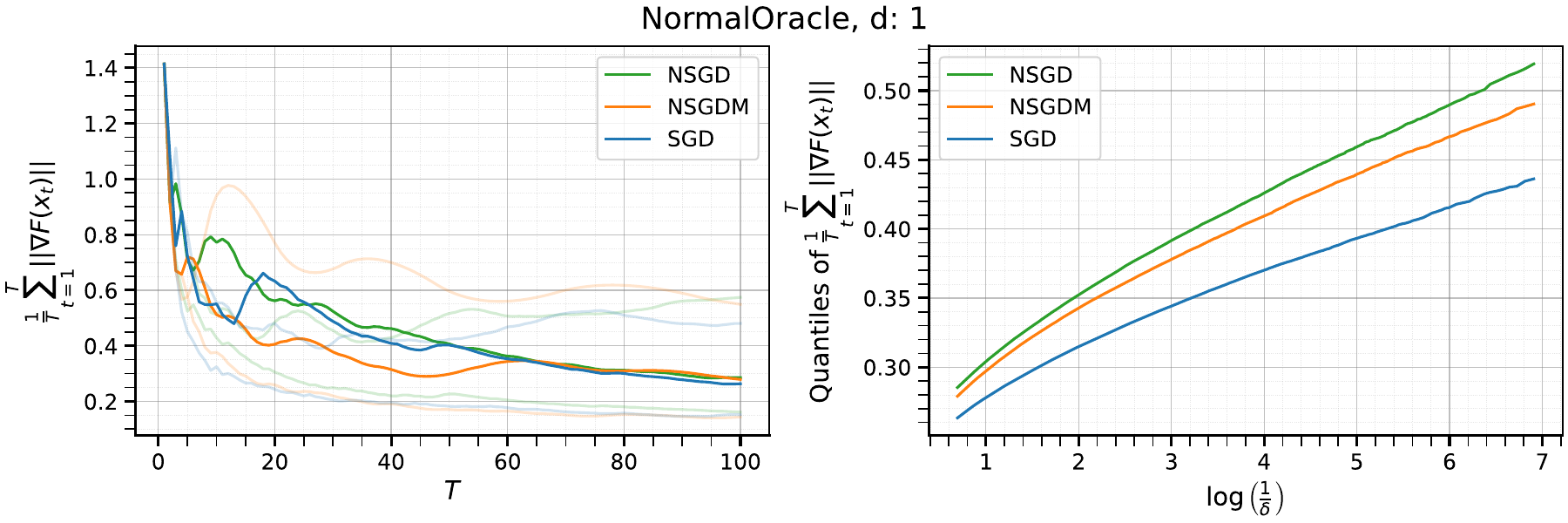}
	\caption{Light tail noise, $\xi_{t} \sim \mathcal N(0, I)$. }
	\label{fig:d1k10000var100_normal}
\end{figure}

\begin{figure}
	\centering
	\includegraphics[width=0.9\linewidth]{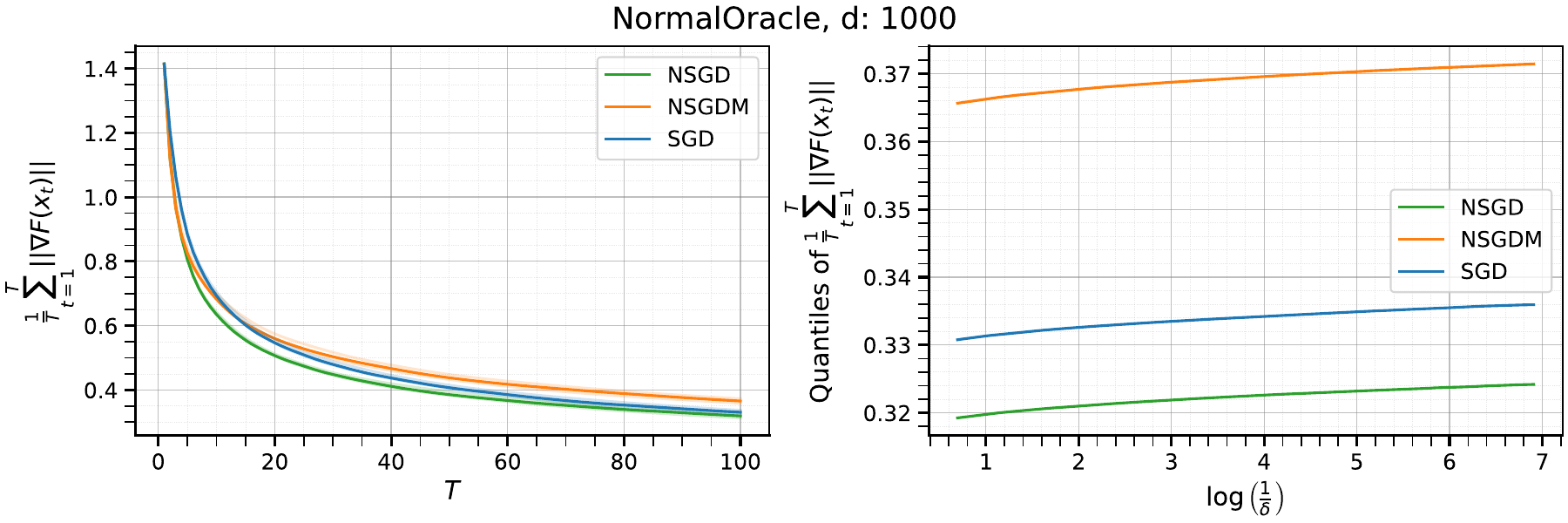}
	\caption{Light tail noise, $\xi_{t} \sim \mathcal N(0, I)$. }
	\label{fig:d1000k10000var100_normal}
\end{figure}

\paragraph{Light tailed noise.} Our results reveal that for the Normal noise distribution, which has light tails, all three algorithms exhibit sub-linear curves \Cref{fig:d1k10000var100_normal,fig:d1000k10000var100_normal}, indicating high probability convergence. However, for Pareto noise \Cref{fig:d1k10000var100_pareto25,fig:d1000k10000var100_pareto25,fig:d1k10000var100_pareto15,fig:d1000k10000var100_pareto15} (particularly with $p = 1.5$ and $p = 2.5$), which corresponds to infinite and finite variance regimes respectively, \sgdName exhibits a super-linear curve, confirming its lack of high probability convergence, consistent with theoretical predictions.

\paragraph{Heavy tailed noise.} Most importantly, we observe that while both \nsgd\ and \nsgdm\ exhibit similar qualitative behaviors when the noise has light tails, \Cref{fig:d1k10000var100_normal,fig:d1000k10000var100_normal}; the quantile dependence on $\log(1/\delta)$ can be super-linear under heavy tailed noise \Cref{fig:d1000k100000var1_pareto25_nsgdm,fig:d1000k100000var1_pareto15_nsgdm}, strongly suggesting that the momentum version of \nsgdName (\nsgdm) may not possess a high probability convergence as \nsgdName.

\begin{figure}
	\centering
	\includegraphics[width=0.9\linewidth]{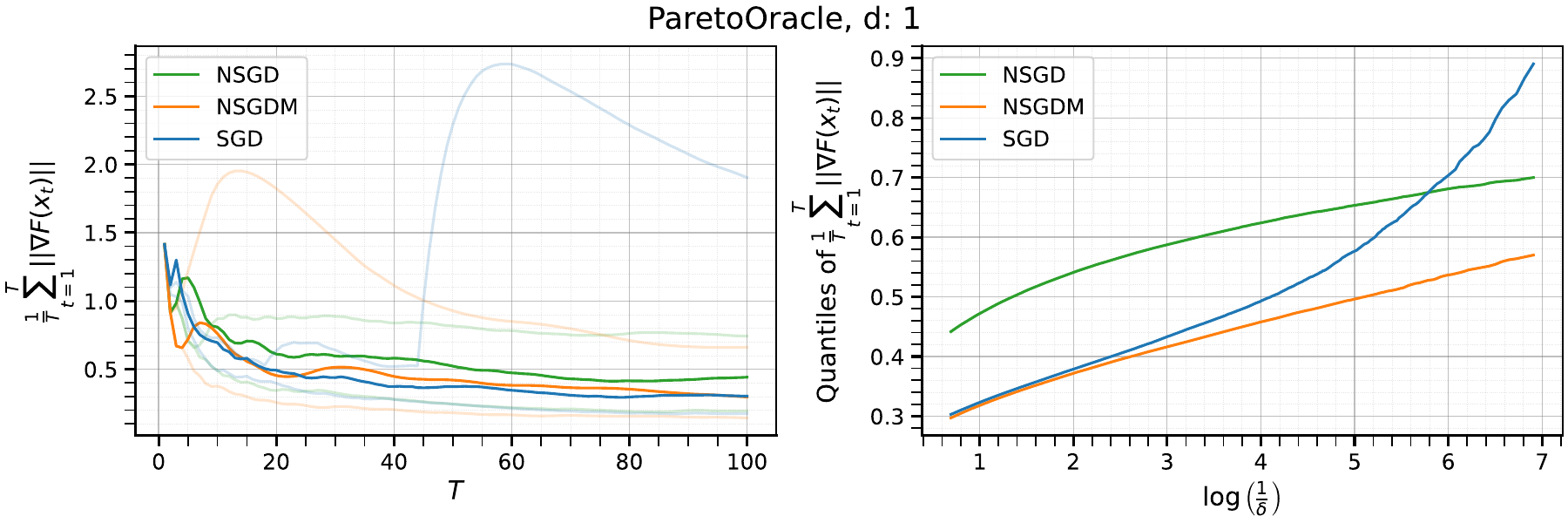}
	\caption{Heavy tailed noise with finite variance. Pareto with $p = 2.5$.}
	\label{fig:d1k10000var100_pareto25}
\end{figure}

\begin{figure}
	\centering
	\includegraphics[width=0.9\linewidth]{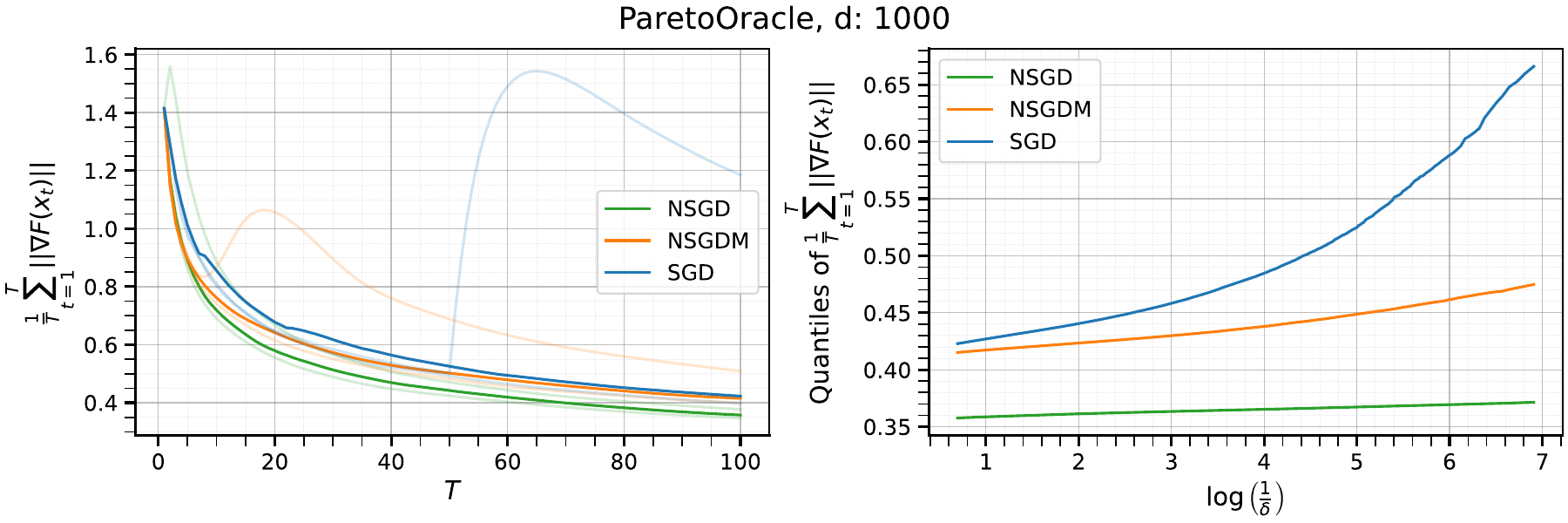}
	\caption{Heavy tailed noise with finite variance. Pareto with $p = 2.5$. See \Cref{fig:d1000k100000var1_pareto25_nsgdm} for the same experiment, but without \sgdName\ on the plot.}
	\label{fig:d1000k10000var100_pareto25}
\end{figure}

\begin{figure}
	\centering
	\includegraphics[width=0.9\linewidth]{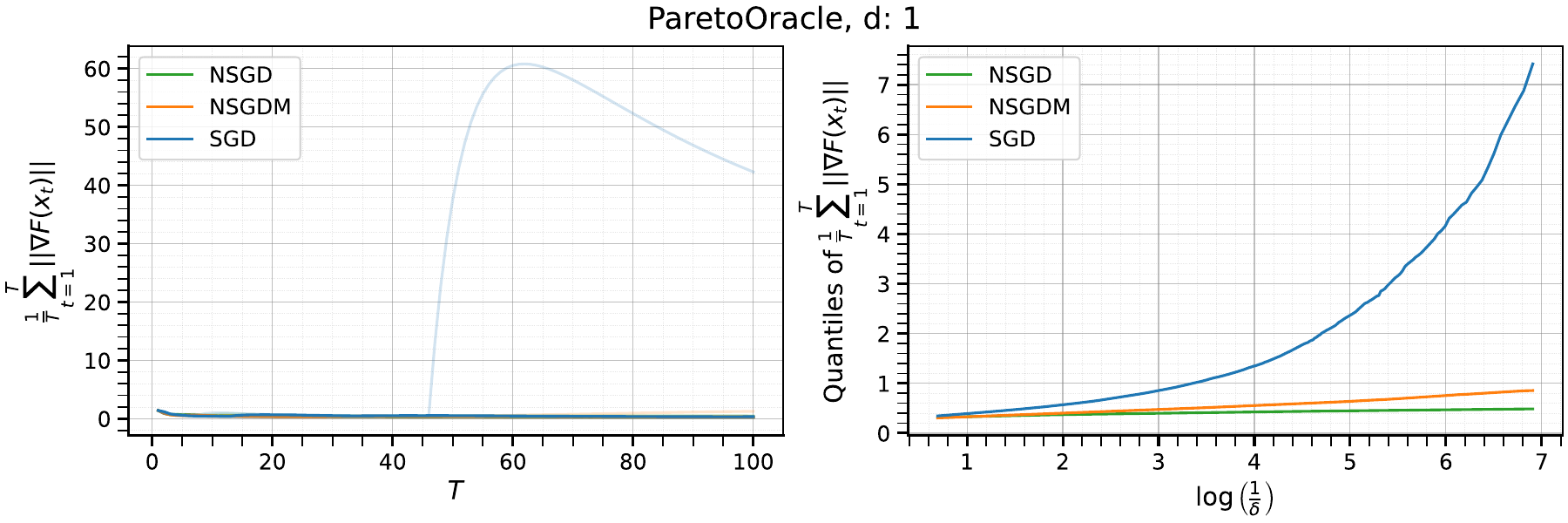}
	\caption{Heavy tailed noise with infinite variance. Pareto with $p = 1.5$.}
	\label{fig:d1k10000var100_pareto15}
\end{figure}

\begin{figure}
	\centering
	\includegraphics[width=0.9\linewidth]{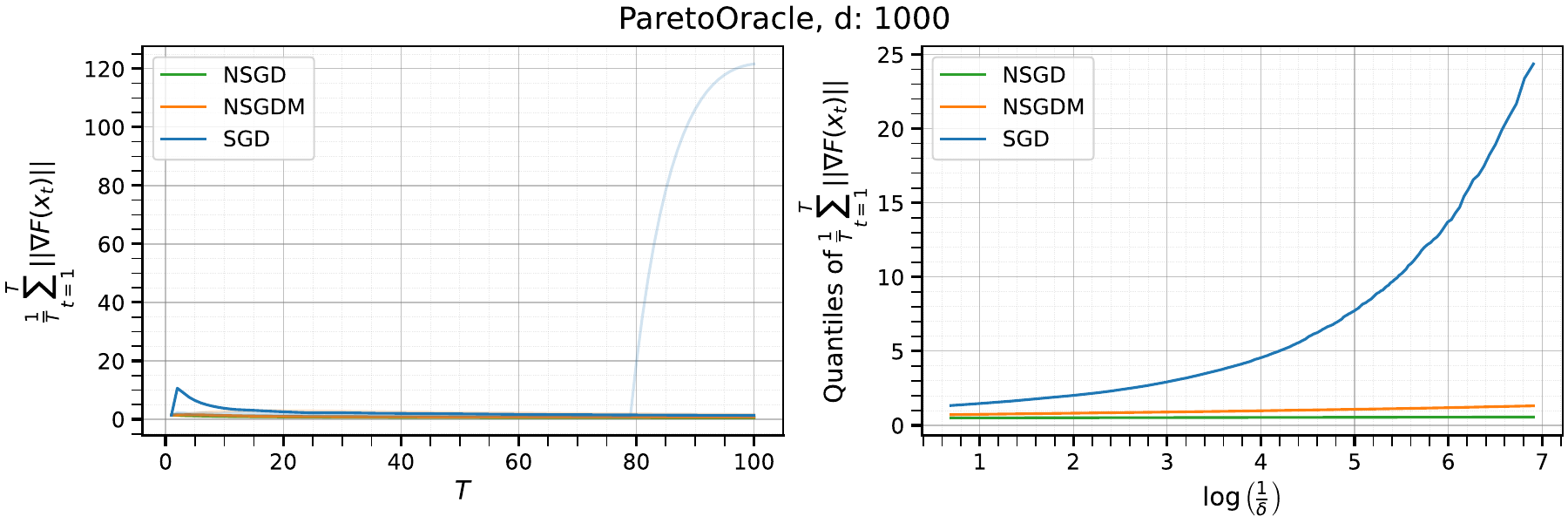}
	\caption{Heavy tailed noise with infinite variance. Pareto with $p = 1.5$. See \Cref{fig:d1000k100000var1_pareto15_nsgdm} for the same experiment, but without \sgdName\ on the plot.}
	\label{fig:d1000k10000var100_pareto15}
\end{figure}

\begin{figure}
	\centering
	\includegraphics[width=0.9\linewidth]{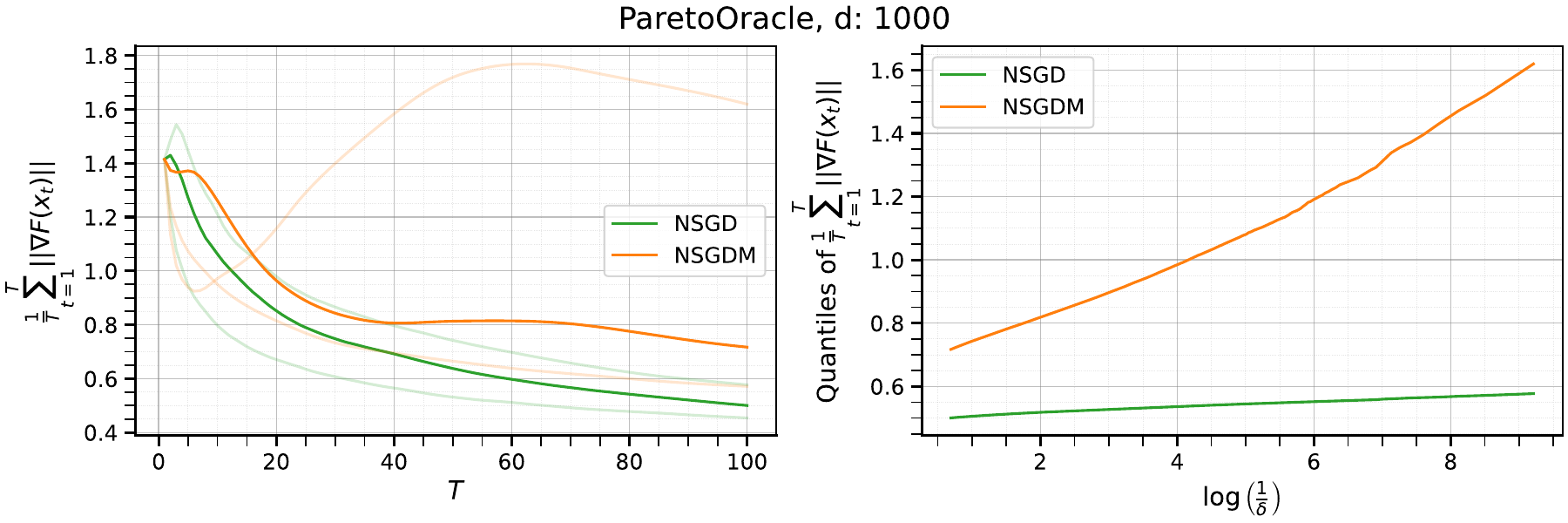}
	\caption{Heavy tailed noise with infinite variance. Pareto with $p = 1.5$.}
	\label{fig:d1000k100000var1_pareto15_nsgdm}
\end{figure}

\end{document}